\documentclass[12pt]{article}
\usepackage[left=2.5cm,right=2.5cm]{geometry}
\usepackage{graphicx}
\usepackage{epstopdf}
\usepackage{amsfonts}
\usepackage{amsmath}
\usepackage{mathrsfs}
\usepackage{amsthm}
\usepackage{subfigure}
 \usepackage{multicol,multicap}
 \usepackage{hyperref}\hypersetup{urlcolor=blue, citecolor=blue}

\newtheorem{theorem}{Theorem}[section]
\newtheorem{lemma}[theorem]{Lemma}

\newtheorem{remark}[theorem]{Remark}

\numberwithin{equation}{section}

\begin{document}


\markboth{Z. Shen $\&$ J. Wei}{Bifurcation Analysis in A Diffusive Mussel-Algae Model with Delay}

\title{Bifurcation analysis in  a diffusive mussel-algae model with delay}

\author{ZUOLIN SHEN\footnote{Email: mathust\_lin@foxmail.com} ~and JUNJIE WEI\footnote{Corresponding author. Email: weijj@hit.edu.cn.}\\
{\small Department of Mathematics,
Harbin Institute of Technology, \hfill{\ }}\\
{\small Harbin, Heilongjiang, 150001, P.R.China\hfill{\ }}\\
}

\maketitle

\begin{abstract}
In this paper, we consider the dynamics of a delayed reaction-diffusion mussel-algae system subject to Neumann boundary conditions. When the delay is zero, we show the existence of positive solutions and the global stability of the boundary equilibrium. When the delay is not zero, we obtain
the stability of the positive constant steady state and the existence of Hopf bifurcation by analyzing the distribution of characteristic values. By using the theory of normal form
and center manifold reduction for partial functional differential equations, we derive an algorithm that determines the direction of Hopf bifurcation and the stability of bifurcating periodic solutions. Finally, some numerical simulations are carried out to
support our theoretical results.
\end{abstract}

\section{Introduction}

\noindent Two-component interactions coupled with dispersion and advection have been formulated for explaining pattern formation \cite{Shigesada-1981, Malchow-1996, Klausmeier-1999, Ainseba-2008}. The researchers' interests are the processes
of generating spatial complexity in ecosystems.
In particular, van de Koppel {\it et al.} [2005] studied the regular spatial patterns in young mussel beds on soft sediments in the Wadden Sea through a spatially explicit model describing changes in local population biomass
of algae and mussels.
The model considered the dispersal effect of the mussel and
the advection effect by the tidal current for the algae but ignored the dispersal effect for the latter.
The simulations have shown that the coupling between dispersion and advection can lead to spatial patterns.
A successful model deserves a further study, such as the implications of advection caused by tidal flow \cite{Sherratt-2016},
kinetic behavior of the patterned solutions \cite{WLS-2009}, interactions between different forms of self-organization \cite{Koppel-2008,Liu-2013, Liu-2014}.

In 2015, based on the field experiment consisting of a young mussel bed on a homogeneous substrate covered by a relatively quiescent layer of marine water in which advection was minimized as much as possible (for more details about the experiment, see \cite{Koppel-2008, Liu-2013}), Cangelosi {\it et al.} [2015] extended the dispersion-advection system in the case of replacing the advection term by a lateral diffusive one:
\begin{equation}\label{eq_1}
\begin{cases}
\cfrac{\partial M}{\partial t}&=D_{M}\Delta M +e c M A -d_{M}\cfrac{k_{M}}{k_{M}+M}M,  \\
\cfrac{\partial A}{\partial t}&=D_{A}\Delta A+(A_{up}-A)f-\cfrac{c}{H}MA,\\
\end{cases}
\end{equation}
where $M=M(x,t)$ is the mussel biomass density on the sediment, $A=A(x,t)$ is the algae concentration in the lower water layer overlying the mussel bed, $x\in \Omega$ is spatial variable, and $\Omega$ is a bounded domain in $\mathbb{R}^n$ with a smooth boundary $\partial{\Omega}$. Here, $e$ is
a conversion constant relating ingested algae to mussel biomass production, $c$ is the consumption constant,
$d_{M}$ is the maximal per capital mussel mortality rate, $k_{M}$ is the value of $M$ at which mortality is
half-maximal, $A_{up}$ describes the uniform concentration of algae in the upper reservoir water layer, $f$ is the rate of exchange between the lower and upper water layers, $H$ is the height of the lower water layer,
and $D_{M}$ and $D_{A}$ are the diffusion coefficients of the mussel and algae respectively.

We shall introduce the following dimensionless change of variables:
$$
\begin{array}{l}
m=\cfrac{M}{k_{_M}},~a=\cfrac{A}{A_{up}},~\omega=\cfrac{c k_{_M}}{H},~\hat{t}=d_{_M}t,~\alpha=\cfrac{f}{\omega},\\
r=\cfrac{e c A_{up}}{d_{_M}},~\gamma=\cfrac{d_{_M}}{\omega},~d=\cfrac{D_{_M}}{\gamma D_{_A}},~\hat{x}=x\sqrt{\cfrac{\omega}{D_{_A}}},\\
\end{array}
$$
then we have
\begin{subequations}\label{eq_ma}
	\begin{equation}\label{eq_ma_1}
	\begin{cases}
	\cfrac{\partial m}{\partial t}=d\Delta m+ rma -\cfrac{m}{1+m},\\
	\gamma  \cfrac{\partial a}{\partial t}=\Delta a+\alpha(1-a)-ma.
	\end{cases}
	\end{equation}
	For simplicity, we have removed the `~$\hat{}$~'.
	
	We point out that most studies of system \eqref{eq_ma_1} concentrate on the formation of patterns and numerical bifurcation, see for examples \cite{Koppel-2008, WLS-2009, Liu-2012, Sherratt-2013}.
	We shall investigate the periodic solutions bifurcated from the constant coexistence steady state. The dynamics near the bifurcation point can well explain the periodicity of population in predator-prey systems. For further mathematical analysis, we supplement system \eqref{eq_ma_1} with the following initial-boundary value conditions:
	\begin{equation}\label{eq_ma_2}
	\begin{array}{l}
	\partial_{\nu}m=\partial_{\nu}a=0, ~x\in \partial\Omega, ~t>0,\\
	m(x,0)=m_{0}(x)\geq 0,~a(x,0)=a_{0}(x)\geq 0, ~x\in\Omega.
	\end{array}
	\end{equation}
\end{subequations}

Time delay has been commonly used in modeling biological systems and can significantly change the dynamics of these systems \cite{Volterra-1928, Wangersky-1957, Dunkel-1968, Freedman-1992, Campbell-1999, SWH-2004, CSW-2013, XXF-2017}. In a predator-prey system, we assume that the prey will die soon after being captured by the predator, while the predator needs a certain period to convert the prey into its energy. Therefore, in the equation of prey, the functional response is not affected by the time delay, while in the equation of predator, the current number of predators dependents on the number of prey present at some previous time.
In this article, we consider the following delayed mussel-algae system:
\begin{equation}\label{eq_ma_tau}
\begin{cases}
\cfrac{\partial m(x,t)}{\partial t}=d\Delta m(x,t)+m(x,t)\left(ra(x,t-\tau) -\cfrac{1}{1+m(x,t-\tau)}\right), &x\in\Omega, ~t>0,\\
\gamma  \cfrac{\partial a(x,t)}{\partial t}=\Delta a(x,t)+\alpha(1-a(x,t))-m(x,t)a(x,t),&x\in\Omega, ~t>0,\\
\partial_{\nu}m=\partial_{\nu}a=0, &x\in\partial\Omega, ~t>0,\\
m(x,t)=m_{0}(x,t)\geq 0,~a(x,t)=a_{0}(x,t)\geq 0, &x\in\Omega, ~-\tau\leq t\leq 0,
\end{cases}
\end{equation}
where $\tau$ is the digestion period of mussel and the mortality of mussels depends on the state whether they have eaten in the past.
The homogeneous Neumann boundary condition implies that there is no population movement across the boundary $\partial\Omega$.

Define the real-value Sobolev space
$$ X:=\left\{(u,v)\in H^2(\Omega)\times H^2(\Omega)|\partial_{\nu}u=\partial_{\nu}v=0, x\in \partial\Omega\right\},
$$
and its complexification $X_{\mathbb{C}}:=X\oplus iX=\{x_1+i x_2|x_1,x_2 \in X\}$ with a complex-valued $L^2$ inner product $<\cdot,\cdot>$ which defined as
$$
<U_1,U_2>=\int_{\Omega}(\bar{u}_1u_2+\bar{v}_1v_2)dx,
$$
where $U_i=(u_i, v_i)^T \in X_{\mathbb{C}}, i=1,2$.

The system \eqref{eq_ma_tau} always has a non-negative constant solution $E_0(0,1)$, which is a boundary equilibrium corresponding to bare sediment where no mussels exist. Biologically, we would like to see the coexistence state corresponding to a positive equilibrium. In order for this to happen, we make the following assumption:
$$
\textsc{(H1)}~~~~~~\qquad~~~ 0<\alpha<1<r<\alpha^{-1}.~~~~~~~~
$$
Then the system has a unique constant positive equilibrium $E_*(m^*,a^*)$ with $m^*=\cfrac{\alpha (r-1)}{1-\alpha r}$ and $ a^*=\cfrac{1-\alpha r}{r(1-\alpha)}$.

The main work of this article is the proof of global existence and boundedness of solutions and
a detailed bifurcation analysis about the positive equilibrium.
In the stability analyses to follow, we first employ $r$ as a bifurcation parameter and consider the Hopf bifurcation of system \eqref{eq_ma} at the positive equilibrium. Then for functional differential system \eqref{eq_ma_tau}, we show the existence of Hopf bifurcation caused by time delay $\tau$. Moreover, we give the direction and stability of bifurcating periodic solutions.

The organization of the remaining part is as follows.
In Section 2, we prove the wellposedness (existence, uniqueness, and positivity) of
solutions to system \eqref{eq_ma}. We also show the linear stability analysis of the positive constant steady state in this section.
In Section 3, for system \eqref{eq_ma_tau}, we consider the existence of Hopf bifurcation with delay as the bifurcation parameter.
In Section 4, we give the direction of Hopf bifurcation and the stability of the bifurcating periodic solutions by applying the normal form method and center manifold theory for partial functional differential equations.
Section 5 is devoted to numerical simulations.
\section{Existence and linear stability analysis for model without delay}
In this section, we mainly focus on the analysis of model \eqref{eq_ma}. We first prove the existence and boundedness of the unique positive solution
by using the method of upper-lower solution and strong maximum principle. Then we show the global attractivity of boundary equilibrium.
Finally, we investigate the Hopf bifurcation induced by the rescaled capture rate $r$ and Turing bifurcation induced by the predator diffusion rate $d$.
\subsection{Existence and boundedness}
The global existence of the solutions for the initial value problem \eqref{eq_ma} is proved in this subsection.
\begin{theorem}\label{theorem-boundedness}
	Assume that $\alpha$, $\gamma$, $r$ and $d$ are all positive, the initial data $\left(m_0(x), a_0(x)\right)$ satisfies $m_{0}(x)\geq 0, a_{0}(x)\geq 0$, and $ m_{0}(x) \not\equiv 0, a_{0}(x) \not\equiv 0$ for $x\in\Omega$. Then
	\begin{enumerate}
		\item  The system \eqref{eq_ma} has a unique nonnegative solution $(m(x,t), a(x,t))$ satisfying
		$$
		0<m(x,t), \quad 0<a(x,t)\leq \max\left\{\|a_0\|_{\infty}, 1\right\}, \quad x\in \bar{\Omega}, t> 0,
		$$
		where $\|\phi\|_{\infty}=\sup_{x\in\overline{\Omega}} \phi(x)$.
		
		\item  If $0<r<1$ and $0<\alpha r<\cfrac{1}{2}$, then the first component $m(x,t)$ of the solutions of system \eqref{eq_ma} satisfies the following estimate
		$$
		\limsup_{t\rightarrow \infty} m(x,t)\leq  1,  \quad x\in \bar{\Omega}.
		$$
	\end{enumerate}
	
\end{theorem}

\begin{proof}
	Let $(m(t),~a(t))$ be the unique solution of the following ODE system
	\begin{equation}\label{eq_ode}
	\begin{cases}
	\cfrac{\text{d}m}{\text{d}t}=rma -\cfrac{m}{1+m},\\
	\gamma \cfrac{\text{d}a}{\text{d}t}=\alpha(1-a)-ma,  \\
	m_0=\sup_{x\in\overline{\Omega}}m_0(x), ~~a_0=\sup_{x\in\overline{\Omega}}a_0.
	\end{cases}
	\end{equation}
	
	Note that (2) is a mixed quasi-monotone system. Hence $(0,0)$ and $(m(t),a(t))$
	are the lower-solution and upper-solution of (2) respectively. From Theorem 3.3 (Chapter 8, page 400) in [Pao, 1992], we know that system (2) has a unique solution
	$(m(x,t),a(x,t))$ which satisfies
	$$
	0\leq m(x,t)\leq m(t),~~0\leq a(x,t)\leq a(t),~t\geq0.
	$$
	By the strong maximum principle for parabolic equations and the comparison principle, we can easily have that $0<m(x,t)$, $0<a(x,t)\leq \max
	\left\{\|a_0\|_{\infty}, 1\right\}$.
	
	To prove the boundedness of $m(x,t)$, we only need to prove that $m(t)$ is bounded since $m(t)$ is a upper-solution of $m(x,t)$, to show this, we first justify two claims.
	
	\textit{\textit{Claim} 1}: For any $T>0$, there exists a $t_1>T$ such that $m(t_1)<1$.
	If not, we assume that $m(t)\geq 1$ holds for all $t>T$. Let $w(t)=\frac{1}{\gamma}m(t)+ra(t)$, then we have
	$\frac{\text{d}w}{\text{d}t}
	\leq\frac{1}{\gamma}\left( r\alpha -\frac{1}{2}\right)<0
	$,
	which indicates $w(t)\rightarrow -\infty$ as $t\rightarrow \infty$ and this contradicts the definition of $w(t)$.
	It follows from part \textit{(1)} that for any $a_0>0$ and $\varepsilon_0>0$, there exists a $t_0>0$ such that $a(t)\leq 1+\varepsilon_0$ for $t\geq t_0$.
	Without loss of generality, we assume that $t_1>t_0$.
	
	\textit{\textit{Claim} 2}: There exists a $t_2>t_1$, such that $m(t)\leq 1$ for all $t>t_2$. To show this, let $a=f^1(m), a= f^2(m)$
	be the nullclines of $m$ and $a$ in the first quadrant, respectively. Then we have
	\begin{equation}\label{f^1-f^2}
	\begin{array}{ll}
	f^1(m)-f^2(m)=\cfrac{1}{r(1+m)}-\cfrac{\alpha}{\alpha+m}>0,
	\end{array}
	\end{equation}
	which implies that the $a-$nullcline is below the $m-$nullcline.
	On the other hand,
	for $t>0$ with $a(t)=f^1(m(t))$, we have
	\begin{equation}\label{da_dt}
	\begin{array}{ll}
	\gamma\cfrac{\text{d}a}{\text{d}t}
	=\alpha\left(1-a(t)\right)-m(t)a(t)
	=\cfrac{\alpha}{r}(r-1)<0.
	\end{array}
	\end{equation}

	Let $\phi_t$ be the trajectory of Eq.\eqref{eq_ode} with the initial value $\phi_0=(m_0,a_0)$ at $t=0$, and denote
	$$\begin{array}{ll}
	\Omega_1=\big\{(m,a): 0<m<1,~0<a<\min\{1+\varepsilon_0,f^1(m)\}\big\},\\
	\Omega_2=\big\{(m,a): 1<m,~0<a<\min\{1+\varepsilon_0,f^1(m)\}\big\},\\
	\Omega_3=\big\{(m,a): 0<m,~f^1(m)<a\big\}.
	\end{array}
	$$
	From \eqref{f^1-f^2}, \eqref{da_dt} and \textit{Claim 1}, we know that $\Omega_1$ is an invariant region. In addition, for any $\phi_0\in \Omega_1\cup\Omega_2$, we have $\phi_t\in \Omega_1$ for $t>t_1$.
	If $\phi_0\in \Omega_3$, then $\frac{\text{d}m}{\text{d}t}> 0, \frac{\text{d}a}{\text{d}t}< 0$. Moreover, there exists a $t'_2$ such that $\phi_t$ meets the $m$-nullcline at $t=t'_2$ and then enter the region $\Omega_1\cup\Omega_2$ (otherwise, $m(t)\rightarrow \infty$ as $t\rightarrow \infty$ which contradicts \textit{Claim 1}), and eventually reach the invariant region $\Omega_1$. See Fig.1 for geometric interpretations.
	
	This completes the proof.
\end{proof}

\begin{figure}[ht!]
\centering
\includegraphics[width=4.5in]{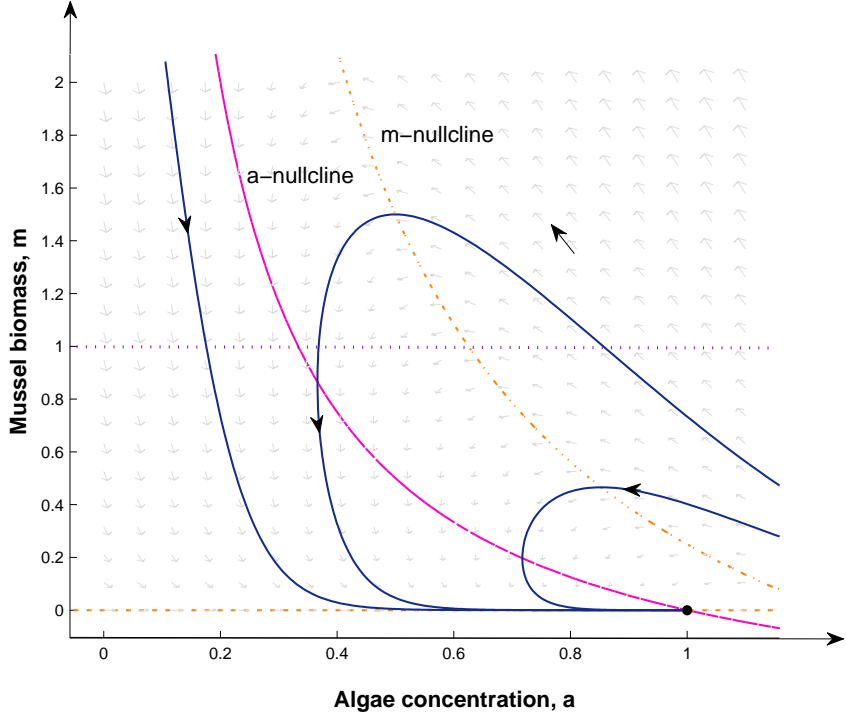}
\caption{Basic phase portrait of \eqref{eq_ode} with $0<r<1$ and $0<\alpha r<\frac{1}{2}$. The dashed-dotted curve is the
$m$-nullcline $a_{_f}= f^1(m)$, the dashed line is the $a$-nullcline $a_{_h}=f^2(m)$, the horizontal dot curve is $m=1$. The parameters used are given by $r=0.8, \alpha=0.5, \gamma=8$.}
\label{phase}       
\end{figure}

\subsection{Global stability of boundary equilibrium}
In this subsection, we shall prove the global stability of the boundary equilibrium $E_0(0,1)$ for the system \eqref{eq_ma} under some additional assumptions.
\begin{theorem}\label{boundary}
	Assume that $\alpha$, $\gamma$, $r$ and $d$ are all positive and the initial data $(m_0(x), a_0(x))$ satisfies the hypotheses of Theorem \ref{theorem-boundedness}. Then
	\begin{enumerate}
		\item If $0<r<1$, then $E_0(0,1)$ is locally asymptotically stable.
		\item If $r>1$, then $E_0(0,1)$ is unstable.
		\item If $0<r<\cfrac{1}{2}$ and $0<\alpha r<\cfrac{1}{2}$, then $E_0(0,1)$ is globally asymptotically stable.
	\end{enumerate}
\end{theorem}
\begin{proof}
	The proof of \textit{(1)} and \textit{(2)}
	can be found in Section \ref{linear and Hopf} in which spatial domain $\Omega$ can work for arbitrary higher dimension.
	Next, we use the Lyapunov functional to prove global attractivity.
	
	Define
	$$ V(m,a)=\gamma r \int_\Omega \int_1^a\cfrac{\xi-1}{\xi}\text{d}\xi\text{d}x+\int_\Omega m\text{d}x.$$
	Then
	$$\begin{array}{ll}
	\dot{V}(m,a)&=\gamma r  \displaystyle \int_\Omega \cfrac{a-1}{a}a_t\text{d}x+ \displaystyle \int_\Omega m_t\text{d}x\\
	&=-r \displaystyle \int_\Omega \cfrac{1}{a^2}|\nabla a|^2\text{d}x-\alpha r  \displaystyle \int_\Omega \cfrac{(1-a)^2}{a}\text{d}x+ \displaystyle \int_\Omega\left(r-\cfrac{1}{1+m}\right)m\text{d}x.
	\end{array}$$
	From Theorem \ref{theorem-boundedness}, we have that $m(x,t)\leq 1$ when $t\geq t_2$ if $0<r<\cfrac{1}{2}$ and $0<\alpha r<\cfrac{1}{2}$. Clearly, $r-\cfrac{1}{1+m}\leq r-\cfrac{1}{2}< 0$ when $t\geq t_2$. Hence, we have $\dot{V}(m,a)\leq 0$ when $t\geq t_2$. Moreover, $\dot{V}(m,a)=0$ implies that $a=1$, and $m=0$ or $m=\cfrac{1}{r}-1$. From the LaSalle invariance principle, we have
	$$\omega(\phi_0)\subset \left\{(0,1),(\cfrac{1}{r}-1,1)\right\},$$
	where $\omega(x)$ is the $\omega-$limit set of $x$.
	Note that $\lim_{t\rightarrow \infty} m(x,t)\leq 1$, then
	$(\cfrac{1}{r}-1,1)\notin \omega(\phi_0)$. Hence, we have
	$$\omega(\phi_0)=\left\{(0,1)\right\},$$
	which is the desired result.
\end{proof}

\subsection{Linear stability and Hopf bifurcation}\label{linear and Hopf}
In this subsection, we shall investigate the linear stability and Hopf bifurcation of system \eqref{eq_ma}, and restrict the spatial domain $\Omega=(0,l\pi)$ of which the structure of the characteristic values is clear.

Denote $U=(m,a)^{^T}$, then the linearization of system \eqref{eq_ma} is
\begin{equation}\label{linearization}
\Gamma U=D\Delta U+ L_{_E}U,
\end{equation}
where $\Gamma$, $D$, $L_{_E}$ are defined as
$$\begin{array}{l}
\Gamma=\left(\begin{array}{cc}
1 & 0\\
0 & \gamma \end{array}\right)$$,\quad

D=\left(\begin{array}{cc}
	d & 0\\
	0 & 1\end{array}\right)$$,\quad

L_{_E}=\left(\begin{array}{cc}
	ra -1/(1+m)^2& rm\\
	-a & -(\alpha+m)\end{array}\right).
\end{array}$$

Then the characteristic equation of \eqref{linearization} at the equilibrium points $E_0(0,1)$ and $E_*(m^*,a^*)$ with Neumann boundary conditions can be obtained, and we first show the characteristic equation corresponding to $E_0(0,1)$, namely:
\begin{equation}\label{ce_0}
(\lambda+1-r+d \cfrac{n^2}{l^2})(\gamma \lambda +\alpha+\cfrac{n^2}{l^2})=0,
\end{equation}
where $n\in\mathbb{N}_0:=\mathbb{N}\cup\{0\}$. By straightforward calculations, we obtain the following results: $E_{0}$ is locally asymptotically stable when $0<r<1$, and unstable when $r>1$.
Our main concern is the dynamics of positive equilibrium $E_*(m^*,a^*)$,
the characteristic equation at $E_*(m^*,a^*)$ can be written as
\begin{equation}\label{ce_1}
\gamma \lambda^2+\widetilde{T}_{n}\lambda +\widetilde{D}_{n}=0,~n\in\mathbb{N}_0,
\end{equation}
where
$$\begin{array}{ll}
\widetilde{T}_{n}=(1+\gamma d)\cfrac{n^2}{l^2}+\cfrac{\alpha}{a^*}-\gamma r^2 a^{*2}m^*,\\
\widetilde{D}_{n}=d\cfrac{n^4}{l^4}+\left(\cfrac{d \alpha}{a^*}-r^2 a^{*2}m^*\right)\cfrac{n^2}{l^2}+\alpha r(r-1)a^*.
\end{array}$$
Then characteristic values $\lambda_{n}$ are given by
\begin{equation}\label{eigen_eq}
\lambda_{n}=\cfrac{-\widetilde{T}_{n}\pm \sqrt{\widetilde{T}_{n}^2-4\gamma \widetilde{D}_{n}}}{2 \gamma}.
\end{equation}

In the remaining part of this section, we choose $r$ as our bifurcation parameter and present some necessary conditions for the occurrence of Hopf bifurcation.

It is well known that if the system \eqref{eq_ma} undergoes a Hopf bifurcation at the critical value $r_{_H}$, there exists a neighborhood $\mathscr{N}(r_{_H})$ of $r_{_H}$ such that for any $r\in \mathscr{N}(r_{_H})$, the characteristic equation \eqref{ce_1} has a pair of simple, conjugate complex roots $\lambda(r)=\beta(r)\pm i\omega(r)$ which continuously differentiable in $r$ and satisfy $\beta(r_{_H})=0$, $\omega(r_{_H})>0$, $\beta'(r_{_H})\neq 0$, and all other roots have non-zero real parts. We shall identify the above conditions through the following form:
\begin{equation}\label{Tn_Dn}
\begin{array}{ll}
\widetilde{T}_{n}(r_{_H})=0,~~~\widetilde{D}_{n}(r_{_H})>0,~~~\beta^{'}(r_{_H})\neq 0,\\
\widetilde{T}_{j}(r_{_H})\neq 0,~~~\widetilde{D}_{j}(r_{_H})\neq 0,~~~j\neq n.
\end{array}
\end{equation}
Note that if $\textsc{(H1)}$ holds, then $\widetilde{D}_{0}=\alpha r(r-1)a^*>0$. The transversality is proved by the recent work in \cite{SJL-2017}, here we just state the following lemma without proof.

\begin{lemma}\label{trans}
Suppose that $\textsc{(H1)}$ holds. Let $r^*=\frac{1}{4}\alpha^{-1}\left(\alpha+\sqrt{\alpha^2+8\alpha}\right)$.
\begin{enumerate}
	\item  If $1<r_{_H}<r^*$, then $\beta'(r_{_H})>0.$
	\item  If $r^*<r_{_H}<\alpha^{-1}$, then $\beta'(r_{_H})<0.$
\end{enumerate}
\end{lemma}

Note that the transversality can always be satisfied as long as $r_{_H}\neq r^*$. Hence, the determination of Hopf bifurcation points reduces to describe the set
$$S:=\left\{r\in(1,\alpha^{-1})\backslash \{r^*\}: \mbox{ for some } n\in\mathbb{N}_0,\eqref{Tn_Dn} \mbox{ is satisfied}\right\}.$$

Denote $\delta_{0}(r)=\cfrac{1-\alpha r}{1-\alpha}$,
$\rho_{0}(r)=\cfrac{r(1-\alpha)}{\gamma(r-1)}$.
The above analysis permits us to give the following stability results for system \eqref{eq_ma} without diffusion, the graphical results can be seen in Fig.\ref{hopf}.

\begin{theorem}\label{theorem_hopf}
Assume that $\textsc{(H1)}$ is satisfied. For system \eqref{eq_ma} without diffusion,
\begin{enumerate}
	\item  if $\delta_{0}^2(r)<\rho_{0}(r)$, the positive equilibrium $E_*(m^*,a^*)$ of system \eqref{eq_ma} is locally asymptotically stable;
	\item  if $\delta_{0}^2(r)>\rho_{0}(r)$, the positive equilibrium $E_*(m^*,a^*)$ of system \eqref{eq_ma} is unstable;
	\item  if $r_{_H}\in S$ satisfies the equation $\delta_{0}^2(r)=\rho_{0}(r)$, the system \eqref{eq_ma} undergoes a Hopf bifurcation at $r=r_{_H}$ which corresponds to spatially homogeneous periodic solution; the critical curve of Hopf bifurcation is defined by $\delta_{0}^2(r)=\rho_{0}(r)$.
\end{enumerate}
\end{theorem}

\begin{figure}[ht!]
\centering
\subfigure[]{\includegraphics[width=3.2in]{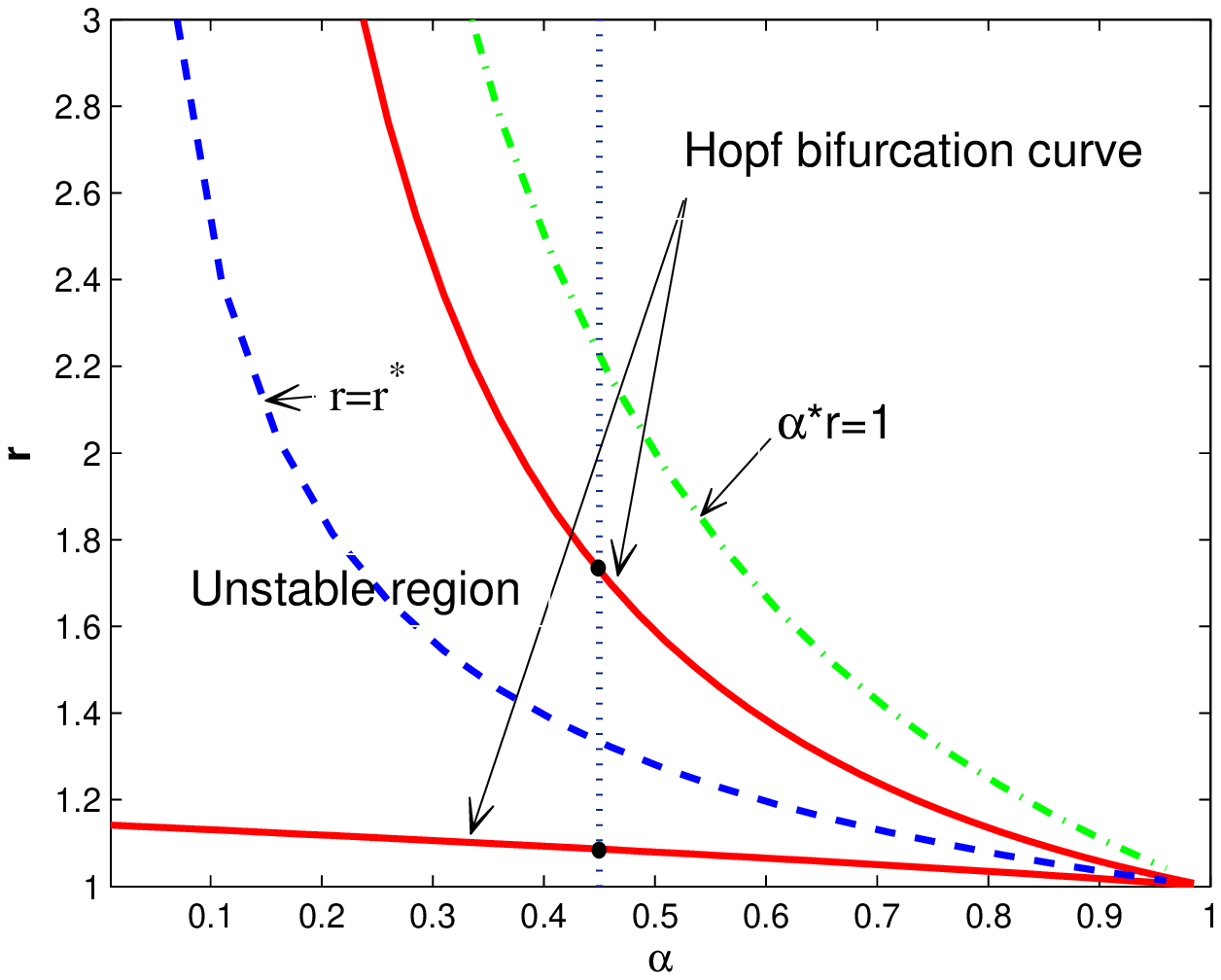}}
\subfigure[]{\includegraphics[width=3.2in]{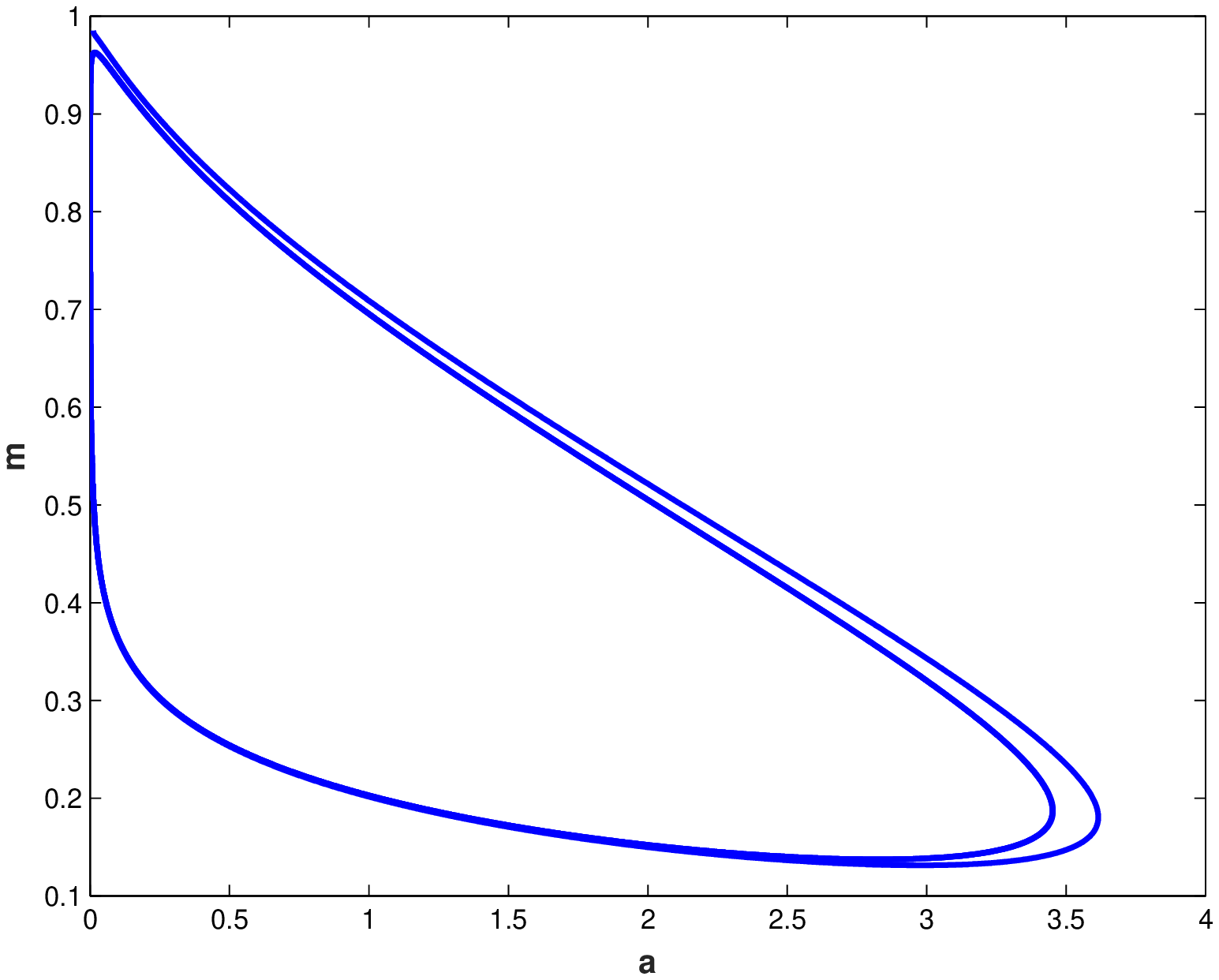}}
\caption{(a) The critical curve of Hopf bifurcation on $\alpha-r$ plane. The vertical dotted
curve is $\alpha=0.45$ and intersects with the Hopf bifurcation curve at two critical points with $r_1=1.0865, r_2=1.7286$ respectively. (b)The limit cycle bifurcated from the positive equilibrium when $r=1.2\in(r_1,r_2)$. The other parameters are chosen as: $\gamma=8, d=0.1$.}
\label{hopf}       
\end{figure}

\begin{remark}\label{E_0}
	When $\tau>0$, the characteristic equation corresponding to $E_0(0,1)$ for system \eqref{eq_ma_tau} has the same expression with Eq.\eqref{ce_0}. Therefore, $E_0(0,1)$ is locally asymptotically stable for any $\tau>0$ when $0<r<1$ (see Fig. \ref{fig3}).
\end{remark}
\begin{remark}
	For system \eqref{eq_ma}, the work of global existence of periodic solutions induced by Hopf bifurcation is still a spot worth studying
	. Our numerical results indicate that there exists at least one periodic solution when $r\in(r_1, r_2)$. Further simulations show that the periodic solution exists globally, see Fig. \ref{global}. The Hopf branch connects two critical points $H_1$ at $r_1$ and $H_2$ at $r_2$.
\end{remark}

\begin{figure}[ht!]
\centering
\subfigure[]{\includegraphics[width=3.2in]{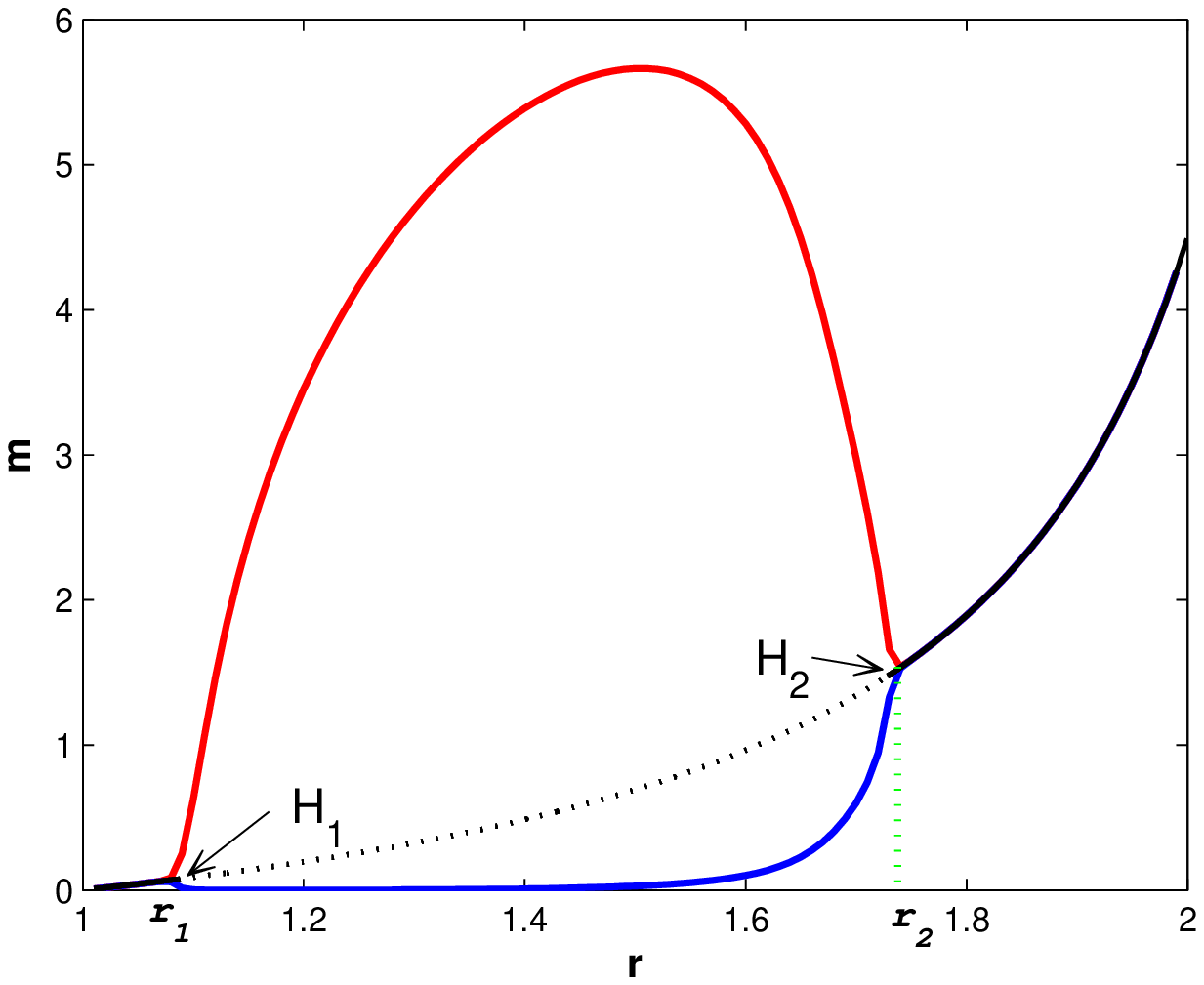}}
\subfigure[]{\includegraphics[width=3.2in]{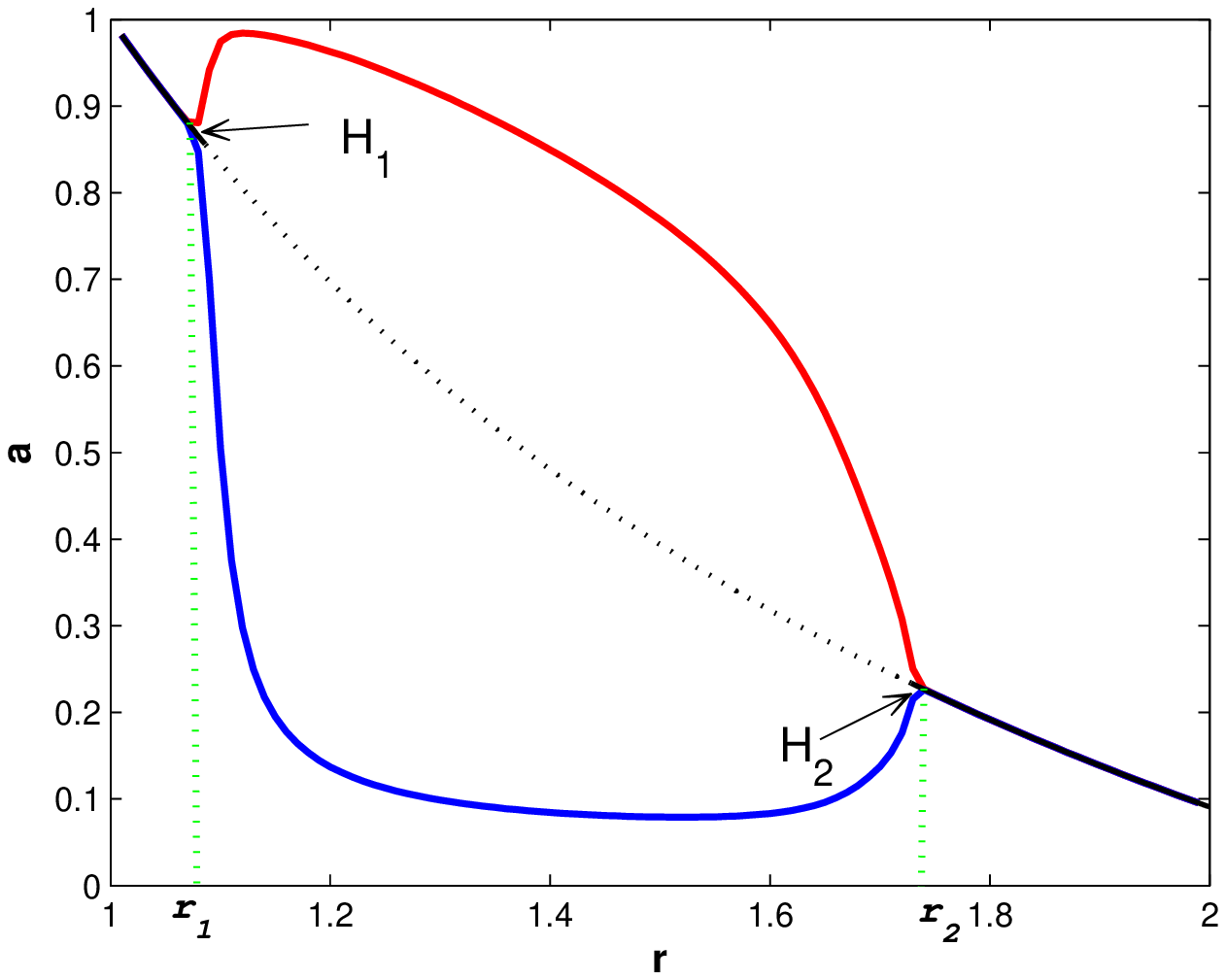}}
\caption{Simulations of global existence of periodic solutions for system \eqref{eq_ma}. Solid lines mark stable portions of the
branch. Black lines represent the homogeneous equilibrium, red lines and blue lines represent maximum
and minimum amplitude, respectively. (a) The amplitude of mussel biomass $m$. (b) The amplitude of algae concentration $a$.  The parameters are chosen as: $\gamma=8, \alpha=0.45, d=0.1$.}
\label{global}       
\end{figure}

\subsection{Turing instability}
In this subsection, we shall consider the Turing instability driven by diffusion. The non-equilibrium phase transition corresponding to the Turing bifurcation is the transformation from uniform steady state to spatial periodic oscillating state. Moreover, the system should be stable to homogeneous perturbations, and unstable to nonhomogeneous ones. To ensure that our stability analysis is valid, we make the following assumption:
$$
\textsc{(H2)}~~~~~~\qquad~~~ \delta_{0}^2(r)-\rho_{0}(r)<0.~~~~~~~~
$$
The following lemma is a trivial work.
\begin{lemma}\label{lemma_gamma_d}
	Let $g(r)=m^*(d\gamma\rho_{0}-\delta_{0}^2)$, $\Lambda=g^2(r)-4d\widetilde{D}_0$. Suppose that $\textsc{(H1)}$ and $\textsc{(H2)}$ holds. Then the positive equilibrium $E_*(m^*,a^*)$ of system \eqref{eq_ma} is locally asymptotically stable if either of
	$I_1$ or $I_2$ holds,
	where $I_1$, $I_2$ are given respectively by
	\begin{equation*}\begin{split}
	(I_1)~~ g(r)\geq0,~~~~
	(I_2)~~ g(r)<0 ~\text{and} ~\Lambda<0.
	\end{split}
	\end{equation*}
\end{lemma}
\begin{proof}
	The sign of $\widetilde{D}_n$ will be determined by the following arguments of $g(r)$ and $\Lambda$:\\
	$(i_1)$:~ ~If $g(r)\geq0$, then
	$$\widetilde{D}_n\geq \widetilde{D}_0>0 \quad  \mbox{for all} ~n\in\mathbb{N}_0;$$
	$(i_2)$:~ ~If $g(r)<0$, and $\Lambda<0$, then
	$$\widetilde{D}_n>0 \quad  \mbox{for all} ~n\in\mathbb{N}_0.$$
	
	Combined with the first case of Thoerem \ref{theorem_hopf}, the Lemma \ref{lemma_gamma_d} follows immediately.
\end{proof}

Lemma \ref{lemma_gamma_d} indicates there is no diffusion-driven Turing instability under $(I_1)$ or $(I_2)$, in this case, diffusion
does not change the stability of the positive equilibrium. Recall that $\widetilde{T}_{n}>\widetilde{T}_{0}>0$ for $n\in\mathbb{N}$, the positive equilibrium in the nonhomogeneous
case changes its stability only when $\widetilde{D}_{n}$ changes sign
from positive to negative. Then a requirement for the occurrence of a Turing instability is the satisfaction of the condition $\widetilde{D}_{n_{c}}=0$, and the critical wave number $k_c$ can be obtained from
\begin{equation}\label{eq_kc}
k^2_c:=\cfrac{n_{c}^2}{l^2}=\cfrac{1}{2d}\left(\cfrac{m^*}{(1+m^*)^2}-\cfrac{d \alpha}{a^*}\right),
\end{equation}
and the necessary condition of Turing instability can be derived by:

\begin{equation}\label{eq_turing}
\alpha d^2 r^2\nu_{0}+\alpha (r-1)^2\mathcal{R}^{-1}_{0} - 2d(r-1)(2-\alpha r)=0,
\end{equation}
where $\nu_{0}=\cfrac{(1-\alpha)^3}{(1-\alpha r)^2}$. Notice that \eqref{eq_turing} is independent of the wave number. The critical curves of Turing bifurcation defined by
formula \eqref{eq_turing} can be seen in Fig.\ref{fig-T}.

\begin{figure}[ht!]
\centering
  \includegraphics[width=5in]{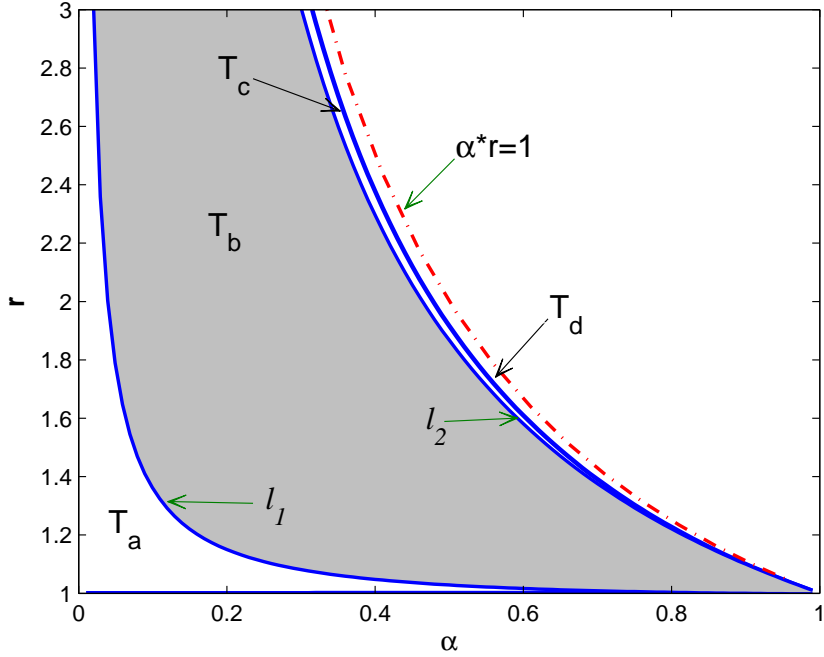}
\caption{The critical curve of Turing bifurcation in $\alpha-r$ plane with values of parameters are chosen as follows: $d=0.01$.}
\label{fig-T}       
\end{figure}

The critical bifurcation curves divide the $\alpha-r$ plane into four regions under the condition $\textsc{(H1)}$ in Fig.\ref{fig-T}. When $d=0.01$, $T_b$ is the only region where Turing instability occurs since $g(r)>0$ in region $T_d$, and $\Lambda<0$ in $T_a$ and $T_c$. In region $T_a$, all the characteristic values of equation \eqref{ce_1} have negative real part, that is, the constant steady state $E_*(m^*,a^*)$ is locally asymptotically stable. When the parameters vary across the curve $l_1$ into the region $T_b$, there is an eigenvalue that moves from the left half complex plane to the right through the origin, and a spatial periodic oscillating state appears from the constant steady state due to the Turing bifurcation. Similarly, when the parameters pass through $l_2$ into region $T_c$ or even $T_d$, the only positive eigenvalue move back to the left half complex plane, the spatial periodic oscillating state disappears, and $E_*(m^*,a^*)$ regains its stability.

\begin{remark}
	According to the research presented by \cite{Turing-1952}, the formation mechanism of Turing pattern is a nonlinear reaction-kinetics process coupled with a special type of diffusion process, and this process requires that the diffusion velocity of activator must be far less than that of the inhibitor in the system. Generally, the interaction-diffusion predator-prey model also follows such a mechanism. By calculating the Jacobian matrix at the equilibrium, we can identify the role of predator and prey in activator-inhibitor systems: in most cases, the prey serves as the ``activator", while predator serves as the ``inhibitor".
	However, under \textsc{(H2)}, according to lemma \ref{lemma_gamma_d}, if Turing instability occurs in model \eqref{eq_ma}, then $d\gamma<1$ must hold, that is, the diffusion velocity of predator must be far less than that of the prey; on the other hand, from the Jacobian matrix $J=(J_{ij})$ at the $E_*(m^*,a^*)$, we know that $J_{11}>0$, $J_{22}<0$, all those indicate that in our model the mussels (predator) play the role ``activator", while algae, the ``inhibitor".
\end{remark}

\section{The existence of Hopf bifurcation induced by delay}\label{stability}
To show the existence
of periodic solutions for system \eqref{eq_ma_tau} with $\tau\geq 0$, we consider the Hopf bifurcation, and we always assume that \textsc{(H1)} and \textsc{(H2)} are satisfied in the remaining part of this section.

Let the phase space $\mathscr{C}:=C([-\tau,0],X_{\mathbb{C}})$ with the sup norm. The linearization of system \eqref{eq_ma_tau} at $E_*(m^*,a^*)$ is given by
\begin{equation}\label{eq41}
\Gamma\dot{U}(t)=D\Delta U(t)+L_*(U_t),
\end{equation}
where $L_*:\mathscr{C}\to X_{\mathbb{C}}$ is defined as
$$
L_*(\phi)=L_1\phi(0)+L_2\phi(-\tau),
$$
with
$$\begin{array}{l}
L_1=\left( \begin{array}{cc}
0 ~&~ 0\\
-a^*  ~&~ -(\alpha+m^*)
\end{array}\right),~~
\quad L_2=\left( \begin{array}{cc}
\cfrac{m^*}{(1+m^*)^2} ~&~ r m^* \\
0 ~&~ 0
\end{array}\right).
\end{array}$$

The corresponding characteristic equations can be deduced by
\begin{equation}\label{c-eq2}
\gamma\lambda^2+T_n\lambda+(B\lambda+M_n)e^{-\lambda\tau}+D_n=0, ~~n\in\mathbb{N}_0,
\end{equation}
where
\begin{equation}\label{TDMB}
\begin{array}{l}
T_n=\alpha+m^*+(1+\gamma d)\cfrac{n^2}{l^2}, \quad  M_n=ra^*m^*(1-\alpha r-ra^*\cfrac{n^2}{l^2}),\\
D_n=d(\alpha+m^*+\cfrac{n^2}{l^2})\cfrac{n^2}{l^2},\quad B=-\gamma r^2{a^*}^2m^*.
\end{array}
\end{equation}
Moreover, $\lambda=0$ is not the root of Eq.\eqref{c-eq2} if the following assumption holds:
\begin{equation*}\textsc{(H3)}~~~~~~~
\begin{cases}
d\gamma\rho_{0}-\delta_{0}^2>0,\\
~\mbox{or}\\
d\gamma\rho_{0}-\delta_{0}^2<0, ~\mbox{and}~~(d\gamma\rho_{0}-\delta_{0}^2)^2-\cfrac{4d\widetilde{D}_0}{m^{*2}}<0.
\end{cases}
\end{equation*}
According to the results in \cite{Ruan-2003}, as parameter $\tau$ varies, the sum
of the orders of the zeros of \eqref{c-eq2} in the open right half plane can change
only if a pair of conjugate complex roots appear on or cross the imaginary axis.
Now we would like to seek critical values of $\tau$ such that there exists a pair of simple purely imaginary eigenvalues. Let $\pm i\omega(\omega>0)$ be
solutions of the $(n+1)$th equation of \eqref{c-eq2}, then
$$-\gamma\omega^2+i T_n \omega+(i\omega B+M_n)e^{-i\omega\tau}+D_n=0.$$
Separating the real and imaginary parts, it follows that
\begin{equation}\label{re-im}
\begin{cases}
M_n\cos\omega\tau +\omega B\sin\omega\tau= \gamma\omega^2-D_n,\\
M_n\sin\omega\tau -\omega B\cos\omega\tau =T_n\omega . \\
\end{cases}
\end{equation}
Then we have
\begin{equation}\label{omega_2}
\gamma^2 z^2+\mathscr{T}_nz+D_n^2-M^2_n=0,
\end{equation}
where $z=\omega^2$ and $\mathscr{T}_n=T_n^2-2\gamma D_n-B^2>0$ is automatically satisfied due to the assumption $\textsc{(H2)}$ .
Hence, if $D_n^2-M_n^2<0$, Eq.\eqref{omega_2} has a unique positive root given by
\begin{equation}\label{z_n}
z_n=\cfrac{1}{2\gamma^2}\left(-\mathscr{T}_n+\sqrt{\mathscr{T}_n^2-4\gamma^2(D_n^2-M_n^2)}\right),
\end{equation}
where
\begin{equation*}\begin{array}{l}
D_n-M_n=d\cfrac{n^4}{l^4}+(d\cfrac{\alpha}{a^*}+r^2 a^{*2}m^*)\cfrac{n^2}{l^2}-\alpha r(r-1)a^*
\to\infty,~\text{as}~n\to\infty,
\end{array}
\end{equation*}
with $D_0-M_0=-ar(r-1)a^*<0$.
Therefore, there exists an integer $N_3\in\mathbb{N}$ such that
\begin{equation*}
D_n-M_n\begin{cases}
<0,~~\text{for}~~0\leq n< N_3,\\
\geq 0,~~\text{for}~~n\geq N_3.
\end{cases}
\end{equation*}

Denote
$$\begin{array}{l}
S_0=\{n\in\mathbb{N}_0|~\mbox{Eq.}\eqref{omega_2} \mbox{ has positive roots under \textsc{(H1)} $\sim$ \textsc{(H3)}}\}.\\
\end{array}$$
Through the analysis above, we know that if $n\in S_0$, then Eq.\eqref{c-eq2} has purely imaginary roots as long as $\tau$
takes the critical values determined by \eqref{re-im}, and those values can be formulated explicitly by
\begin{equation}\label{tau_1}
\tau_{n,j}=\begin{cases}
\cfrac{1}{\omega_n}\left(\arccos\cfrac{(\gamma M_n-BT_n )\omega_n^{2}-M_nD_n}{M_n^2+\omega_n^2B^2}+2j\pi\right),&\sin\omega\tau >0 \\
\cfrac{1}{\omega_n}\left(-\arccos\cfrac{(\gamma M_n-BT_n )\omega_n^{2}-M_nD_n}{M_n^2+\omega_n^2B^2}+2(j+1)\pi\right), &\sin\omega\tau <0
\end{cases},~~j\in\mathbb{N}_0,
\end{equation}
where $\omega_n=\sqrt{z_n}$.

Following the work of \cite{Cooke-1982}, we have
\begin{lemma}\label{trans_tau}
	Suppose that \textsc{(H1)}$\sim$\textsc{(H3)} are satisfied.
	Then
	$$
	\beta'(\tau_{n,j})>0,~~for~j\in\mathbb{N}_0,~n\in S_0,
	$$
	where $\beta(\tau)= ~\textrm{Re} \lambda(\tau)$.
\end{lemma}
\begin{proof}
	Substituting $\lambda(\tau)$ into Eq.\eqref{c-eq2} and taking the derivative with respect to $\tau$ on both side, we obtain that
	\begin{equation*}
	\begin{array}{ll}
	\text{Re}\left(\cfrac{\text{d}\lambda}{\text{d}\tau}\right)^{-1}\Big|_{\tau =\tau_{n,j}}
	&=\text{Re}\left[\cfrac{2\gamma\lambda+T_n+Be^{-\lambda\tau}-\tau(B
		\lambda+M_n)e^{-\lambda\tau}}{\lambda(B\lambda+M_n)e^{-\lambda\tau}}\right]_{\tau =\tau_{n,j}}\\
	&=\text{Re}\left[\cfrac{(2\gamma\lambda+T_n)e^{\lambda\tau}}
	{\lambda(B\lambda+M_n)}+\cfrac{B}{\lambda(B\lambda+M_n)}\right]_{\tau =\tau_{n,j}}\\
	&=\cfrac{2\gamma^2\omega_n^2-2\gamma D_n+T^2_n}{(\gamma\omega_n^2-D_n)^2+\omega_n^2 T_n}+\cfrac{-B^2}{B^2\omega_n^2+M_n^2}\\
	&=\cfrac{\sqrt{(T_n^2-2\gamma D_n-B^2)^2-4\gamma^2(D_n^2-M_n^2)}}{B^2\omega_n^2+M_n^2}.
	\end{array}
	\end{equation*}
	Since the sign of $\text{Re}\left(\cfrac{\text{d}\lambda}{\text{d}\tau}\right)$ is same as that of $\text{Re}\left(\cfrac{\text{d}\lambda}{\text{d}\tau}\right)^{-1}$, the lemma follows immediately.
\end{proof}

From \eqref{tau_1}, for a fixed $n\in S_0$, we have that
$$
\tau_{n,j}\leq\tau_{n,j+1},~j\in\mathbb{N}_0.
$$

Let $\tau^*=\min_{n\in S_0}\{\tau_{n,0}\}$ be the smallest critical value. Summarizing the above analysis, we have the following lemma.
\begin{lemma}\label{lemma_tau}
	Assume that \textsc{(H1)}$\sim$\textsc{(H3)} are satisfied. Then
	the $(n+1)$th equation of \eqref{c-eq2} has a pair of simple pure imaginary roots
	$\pm i\omega_n$, and all the other roots have non-zero real parts when $\tau=\tau_{n,j},~j\in\mathbb{N}_0,~n\in S_0$. Moreover, all the roots of Eq.\eqref{c-eq2} have negative real parts for
	$\tau\in[0,\tau^*)$, and for $\tau>\tau^*$ , Eq.\eqref{c-eq2}
	has at least one pair of conjugate complex roots with positive real parts.
\end{lemma}

Lemma \ref{trans_tau} and Lemma \ref{lemma_tau} lead to the following theorem.
\begin{theorem}\label{Hopf_stability}
	Assume that \textsc{(H1)}$\sim$\textsc{(H3)} are satisfied.
	Then system \eqref{eq_ma_tau} undergoes a Hopf bifurcation at the equilibrium $E_*(m^*,a^*)$
	when $\tau=\tau_{n,j}$, for $j\in\mathbb{N}_0,~n\in S_0$. Furthermore, the positive equilibrium $E_*(m^*,a^*)$ of system \eqref{eq_ma_tau} is asymptotically stable for $\tau\in[0,\tau^*)$, and unstable for $\tau>\tau^*$.
\end{theorem}

\section{Direction of Hopf bifurcation and stability of bifurcating periodic solution}

From the discussion in Section \ref{stability}, the system \eqref{eq_ma_tau} undergoes a Hopf bifurcation at $E_*(m^*,a^*)$ when $\tau=\tau^*$. We will study the direction of Hopf bifurcation
and stability of the bifurcating periodic solutions by using the normal form method and center manifold theory (see \cite{Hassard-1981, Wu-1996, Faria-2000} for more details).

Let $\tilde{m}(x,t)=m(x,t)-m^*$, $\tilde{a}(x,t)=a(x,t)-a^*$, $t\mapsto t/\tau$, and drop the tildes for convenience of notation, then we have
\begin{equation}\label{eq_origin}
\begin{cases}
\cfrac{\partial m}{\partial t}=\tau[d\Delta m+r^2a^{*2}m^* m_t(-1)+r m^*a_t(-1)+f_1(m_t,a_t)], & x\in\Omega,~t>0,\\
\gamma \cfrac{\partial a}{\partial t}=\tau[\Delta a-a^*m-(\alpha+m^*)a+f_2(m_t,a_t)], & x\in\Omega,~t>0,\\
\cfrac{\partial m}{\partial \nu}=0,~\cfrac{\partial a}{\partial \nu}=0, & x\in\partial \Omega,~t>0,\\
m(x,t)=m_0(x,t)-m^*,~ a(x,t)=a_0(x,t)-a^*, & x\in\Omega,-1\leq t\leq 0,
\end{cases}
\end{equation}
where
$$
m_t(\theta)=m(x,t+\theta),~a_t(\theta)=a(x,t+\theta),~~\theta\in [-1,0],
$$
and
\begin{equation}\label{f1}
\begin{array}{ll}
f_1(\phi_1,\phi_2)=&r\phi_1(0)\phi_2(-1)-\cfrac{m^*}{(1+m^*)^3}\phi_1^2(-1)+\cfrac{1}{(1+m^*)^2}\phi_1(0)\phi_1(-1)\\
&+\cfrac{m^*}{(1+m^*)^4}\phi_1^3(-1)-\cfrac{1}{(1+m^*)^3}\phi_1(0)\phi_1^2(-1)+\emph{O}(4),\\
f_2(\phi_1,\phi_2)=&-\phi_1(0)\phi_2(0),
\end{array}
\end{equation}
where  $\phi_1, \phi_2\in \mathcal{C}:=C([-1,0],X_{\mathbb{C}})$.

Let $\tau=\tau^*+\epsilon$. Then the system \eqref{eq_origin} undergoes a
Hopf bifurcation at the equilibrium $(0,0)$ when $\epsilon=0$. Hence we can rewrite system \eqref{eq_origin} in an abstract form in the
space $\mathcal{C}$ as
\begin{equation}\label{eq_abs_origin}
\dot{U}(t)=\tilde{D}\Delta U(t)+L_\epsilon(U_t)+F(\epsilon,U_t),
\end{equation}
where $\widetilde{D}=(\tau^*+\epsilon)\Gamma^{-1}D$, and $
L_\epsilon:\mathcal{C}\to X_{\mathbb{C}},~F:\mathcal{C}\to X_{\mathbb{C}}$
are defined by
$$
L_\epsilon(\phi)=(\tau^*+\epsilon)\Gamma^{-1}L_1\phi(0)+(\tau^*+\epsilon)\Gamma^{-1}L_2\phi(-1),
$$
$$
F(\epsilon,\phi)=\Gamma^{-1}(F_1(\epsilon,\phi),~F_2(\epsilon,\phi))^T,
$$
with
$$
(F_1(\epsilon,\phi),~F_2(\epsilon,\phi))=(\tau^*+\epsilon)(f_1(\phi_1,\phi_2),~f_2(\phi_1,\phi_2)),
$$
where $f_1$ and $f_2$ are defined by \eqref{f1}.

The linearized equation of \eqref{eq_abs_origin} at the origin $(0,0)$ is in the following form
\begin{equation}\label{eq55}
\dot{U}(t)=\widetilde{D}\Delta U(t)+L_\epsilon(U_t).
\end{equation}
According to the theory of semigroup of linear operator \cite{Pazy-1983}, we know that the solution operator of \eqref{eq55} is a $C_0$-semigroup, and
the infinitesimal generator $A_{\epsilon}$ is given by
\begin{equation}\label{A}
A_{\epsilon}\phi=\begin{cases}
\dot{\phi}(\theta),  &\theta\in[-1,0), \\
\widetilde{D}\Delta\phi(0)+L_{\epsilon}(\phi), \qquad &\theta=0,
\end{cases}
\end{equation}
with
$$
\text{dom}(A_\epsilon):=\{\phi\in\mathcal{C}: \dot{\phi}\in\mathcal{C}, \phi(0)\in\text{dom}(\Delta),
\dot{\phi}(0)=\widetilde{D}\Delta\phi(0)+L_{\epsilon}(\phi)\}.
$$

In order to study the dynamics near the Hopf bifurcation, we need to extend the domain of solution operator to a space  of some discontinuous. Let
$$
\mathcal{BC}:=\left\{\phi:[-1,0] \rightarrow X_{\mathbb{C}}\big|~ \phi ~\text{is continuous on} [-1,0), \lim_{\theta\rightarrow 0^{-}}\phi(\theta)\in X_{\mathbb{C}} ~ \text{exists}\right\}.
$$
Hence, equation \eqref{eq_origin} can be rewritten as the abstract ODE in $\mathcal{BC}$
\begin{equation}\label{abstract ODE}
\dot{U}_t=A_\epsilon U_t+X_0 F(\epsilon,U_t),
\end{equation}
where
\begin{equation*}
X_0(\theta)=\begin{cases}
0 , &\theta\in[-1, 0),\\
I, \quad &\theta=0.
\end{cases}
\end{equation*}

Let
$$
b_n=\cfrac{\cos (nx/l)}{\|\cos(nx/l)\|},~ ~\beta_n=\{\beta_n^1, \beta_n^2\}=\{(b_n, 0)^{T}, (0, b_n)^{T}\},
$$
where
$$
\|\cos(nx/l)\|=\left(\int_0^{l\pi}\cos^2(nx/l)\text{d}x\right)^{\frac{1}{2}}.
$$

For $\phi=(\phi^{^{(1)}},\phi^{^{(2)}})^{T}\in\mathcal{C}$, denote
$$
\phi_n=\langle \phi,\beta_n\rangle=\left(\langle \phi,\beta_n^1\rangle, \langle \phi,\beta^2_n\rangle\right)^{T},
$$
and define $A_{\epsilon, n}$ as
\begin{equation}\label{An}
A_{\epsilon, n}(\phi_n(\theta)b_n)=\begin{cases}
\dot{\phi}_n(\theta)b_n,& \theta\in[-1,0), \\
\int_{-1}^{0}\text{d}\eta_n(\epsilon,\theta)\phi_n(\theta)b_n ,\qquad &\theta=0,
\end{cases}
\end{equation}
where
$$
\int_{-1}^{0}\text{d}\eta_n(\epsilon,\theta)\phi_n(\theta)=-\cfrac{n^2}{l^2} \widetilde{D}\phi_n(0)+L_{\epsilon,n}(\phi_n),
$$
with
$$
L_{\epsilon, n}(\phi_n)=(\tau^*+\epsilon)\Gamma^{-1}L_1\phi_n(0)+(\tau^*+\epsilon)\Gamma^{-1}L_2\phi_n(-1),
$$
and
\begin{equation*}
\eta_n(\epsilon,\theta)=\begin{cases}\begin{array}{ll}
-(\tau^*+\epsilon)\Gamma^{-1}L_2, & \theta=-1,\\
0, & \theta\in(-1,0),\\
(\tau^*+\epsilon)\Gamma^{-1}L_1-\cfrac{n^2}{l^2}\widetilde{D}, & \theta=0.
\end{array}\end{cases}
\end{equation*}

Now, we introduce the bilinear form $(\cdot,\cdot)$ on $\mathcal{C}^*\times\mathcal{C}$
\begin{equation}\label{bilinear form}
(\psi,\phi)=\sum_{k,j=0}^\infty(\psi_k,\phi_j)_c\int_\Omega b_kb_j\text{d}x,
\end{equation}
where
$$
\psi=\sum_{n=0}^\infty \psi_n b_n\in\mathcal{C}^*,~\phi=\sum_{n=0}^\infty \phi_n b_n\in\mathcal{C},
$$
and
$$
\phi_n\in C:=C([-1,0],\mathbb{R}^2),~~\psi_n\in C^*:=C([0,1],\mathbb{R}^2).
$$
Notice that
$$
\int_\Omega b_kb_j\text{d}x=0~~\mbox{for}~~k\neq j,
$$
then we have
\begin{equation*}
(\psi,\phi)=\sum_{n=0}^\infty(\psi_n,\phi_n)_c|b_n|^2,
\end{equation*}
where $(\cdot,\cdot)_c$ is the bilinear form defined on $C^*\times C$ with the form
\begin{equation}\label{bilinear form_2}
(\psi_n,\phi_n)_c=\psi_n(0)\phi_n(0)-\int_{-1}^0\int_{\xi=0}^\theta\psi_n(\xi-\theta)
\text{d}\eta_n(0,\theta)\phi_n(\xi)\text{d}\xi.
\end{equation}
Let $A^*$ be the adjoint operator of $A_0$ on $\mathcal{C}^*:=C([0,1],X_{\mathbb{C}})$ under the bilinear form \eqref{bilinear form}. Then
\begin{equation*}
A^*\psi(s)b_n=\begin{cases}\begin{array}{ll}
-\dot{\psi}(s)b_n,~& s\in(0,1],\\
\sum_{n=0}^\infty\int_{-1}^0\text{d}\eta_n^T(0,\theta)\psi_n(-\theta)b_n,~&s=0.
\end{array}\end{cases}
\end{equation*}
Let
$$
q(\theta)b_{n_0}=q(0)e^{i\omega_{n_0}\tau^*\theta}b_{n_0},
~q^*(s)b_{n_0}=q^*(0)e^{-i\omega_{n_0}\tau^* s}b_{n_0}
$$
be the eigenfunctions of $A_0$ and $A^*$ corresponding to the eigenvalues $i\omega_{n_0}\tau^*$.
By direct calculations, we have
$$
q(0)=(1, q_1)^{T},~q^*(0)=M(q_2, 1),
$$
where
\begin{equation*}\begin{array}{ll}
q_1=-\cfrac{a^*}{i\gamma\omega_{n_0}+\alpha+m^*+n_0^2/l^2},~~q_2=\cfrac{i\gamma\omega_{n_0}+\alpha+m^*+n_0^2/l^2}{rm^*e^{-i \omega_{n_0}\tau^*}},\\
M=\cfrac{e^{i\omega_{n_0}\tau^*}}{(q_1+q_2)e^{i\omega_{n_0}\tau^*}+\tau^*q_2(r^2a^{*2}m^*+q_1rm^*)}.
\end{array}\end{equation*}
Then we decompose the space $\mathcal{C}$ as follows
$$
\mathcal{C}=P\oplus Q,
$$
where
$$
P=\{zqb_{n_0}+\overline{z}\overline{q}b_{n_0}|z\in\mathbb{C}\},
$$
$$
Q=\{\phi\in\mathcal{C}|(q^*b_{n_0},\phi)=0~\text{and}~(\overline{q}^*b_{n_0},\phi)=0\}.
$$
$P$ is the 2-dimensional center subspace spanned by the basis vectors of the linear operator $A_0$ associated with purely imaginary eigenvalues $\pm i\omega_{n_0}\tau_*$, and $Q$ is the complement space of $P$.

Thus, system \eqref{abstract ODE} can be rewritten as
$$
U_t=z(t)q(\cdot)b_{n_0}+\bar{z}(t)\bar{q}(\cdot)b_{n_0}+W(t,\cdot),
$$
where
\begin{equation}\label{eq58}
z(t)=(q^*b_{n_0}, U_t),~~~W(t,\cdot)\in Q,
\end{equation}
and
\begin{equation}\label{eq59}
W(t,\theta)=U_t(\theta)-2\text{Re}\{z(t)q(\theta)b_{n_0}\}.
\end{equation}
Then we have
\begin{equation}\label{eq510}
\dot{z}(t)=i\omega_0\tau^*z(t)+q^{*}(0)\langle F(0, U_t), \beta_{n_0}\rangle,
\end{equation}
where
$$
\langle F, \beta_{n} \rangle:=(\langle F_1, b_{n}\rangle,\langle F_2, b_{n} \rangle)^T.
$$
It follows from Appendix A of \cite{Hassard-1981} (also see \cite{LSW-1992}), there exists a center manifold $\mathscr{C}_0$
and we can write $W$ in the following form on $\mathscr{C}_0$ nearby $(0,0)$
\begin{equation}\label{eq511}
W(t,\theta)=W(z(t),\bar{z}(t),\theta)=W_{20}(\theta)\frac{z^2}{2}+W_{11}(\theta)z\bar{z}
+W_{02}(\theta)\frac{\bar{z}^2}{2}+\cdots.
\end{equation}

For $U_t\in\mathscr{C}_0$, we denote
\begin{equation*}
F(0, U_t)\mid _{\mathscr{C}_0}=\tilde{F}(0, z, \bar{z}),
\end{equation*}
with
\begin{equation*}
\tilde{F}(0, z, \bar{z})=\tilde{F}_{20}\frac{z^2}{2}+\tilde{F}_{11}z\bar{z}+\tilde{F}_{02}\frac{\bar{z}^2}{2}
+\tilde{F}_{21}\frac{z^2\bar{z}}{2}+\cdots.
\end{equation*}
Therefore, the system restricted on the center manifold is given by
\begin{equation*}
\dot{z}(t)=i\omega_0\tau^*z(t)+g(z,\bar{z}),
\end{equation*}
where
$$g(z,\bar{z})=g_{20}\frac{z^2}{2}+g_{11}z\bar{z}+g_{02}\frac{\bar{z}^2}{2}
+g_{21}\frac{z^2\bar{z}}{2}+\cdots.$$
By direct calculations, we obtain
\begin{align*}
g_{20}=&\tau^*M\int_0^{l\pi}b_{n_0}^3\text{d}x
\left[q_2\left(2 r q_1 e^{-i\omega_{n_0}\tau^*}+\cfrac{2}{(1+m^*)^2}e^{-i\omega_{n_0}\tau^*}-\cfrac{2 m^*}{(1+m^*)^3}e^{-i2\omega_{n_0}\tau^*}\right)-2 \gamma^{-1}q_1\right], \\
g_{11}=&\tau^*M\int_0^{l\pi}b_{n_0}^3\text{d}x
\left[q_2\left(2\text{Re}\left(\big(rq_1+\cfrac{1}{(1+m^*)^2}\big)e^{-i\omega_{n_0}\tau^*}\right)-\cfrac{2 m^*}{(1+m^*)^3}\right)-\gamma^{-1}(q_1+\bar{q}_1)\right],\\
g_{02}=&\tau^*M\int_0^{l\pi}b_{n_0}^3\text{d}x
\left[q_2\left(2 r \bar{q}_1 e^{i\omega_{n_0}\tau^*}+\cfrac{2}{(1+m^*)^3}e^{i\omega_{n_0}\tau^*}-\cfrac{2 m^*}{(1+m^*)^3}e^{i2\omega_{n_0}\tau^*}\right)-2\gamma^{-1} \bar{q}_1\right], \\
g_{21}=&\tau^*M
\left(Q_1\int_0^{l\pi}b_{n_0}^4\text{d}x+Q_2\int_0^{l\pi}b_{n_0}^2\text{d}x\right), \\
\end{align*}
where
\begin{equation*}
\begin{split}
Q_1=~&q_2\left[\cfrac{6m^*}{(1+m^*)^4}e^{-i\omega_{n_0}\tau^*}-\cfrac{2}{(1+m^*)^3}\left(2+e^{-i2\omega_{n_0}\tau^*}\right)
\right], \\
Q_2=~&q_2\bigg[r\left(2W_{11}^{(2)}(-1)+W_{20}^{(2)}(-1)+\bar{q}_1
e^{i\omega_{n_0}\tau^*}W_{20}^{(1)}(0)+2 q_1e^{-i\omega_{n_0}\tau^*}W_{11}^{(1)}(0)\right)\\
&+\cfrac{1}{(1+m^*)^2}\left(2W_{11}^1(-1)+W_{20}^1(-1)
+W_{20}^1(0)e^{i\omega_{n_0}\tau^*}+2W_{11}^1(0)e^{-i\omega_{n_0}
	\tau^*}\right)\\
&-\cfrac{2m^*}{(1+m^*)^3}\left(2e^{-i\omega_{n_0}\tau^*}
W_{11}^{(1)}(-1)+e^{i\omega_{n_0}\tau^*}W_{20}^{(1)}(-1)\right)
\bigg]\\
&-\gamma^{-1}\left(2W_{11}^{(2)}(0)+W_{20}^{(2)}(0)+\bar{q}_1 W_{20}^{(1)}(0)+2q_1W_{11}^{(1)}(0)\right).\\
\end{split}
\end{equation*}
Since $g_{20}$, $g_{11}$ and $g_{02}$ are independent of $W\left(z(t),\bar{z}(t),\theta\right)$, they can be compued by Eq.\eqref{eq510}. In order to get $g_{21}$, we need to compute $W_{20}$ and $W_{11}$. From \eqref{eq59}, we have
\begin{equation}\label{eq512}
\begin{split}
\dot{W}&=\dot{U}_t-\dot{z}qb_{n_0}-\dot{\bar{z}}\bar{q}b_{n_0} \\
&=\begin{cases}
A_0 W-2\text{Re}\{g(z,\bar{z})q(\theta)\}b_{n_0},&
\theta\in[-1,0),\\
A_0 W-2\text{Re}\{g(z,\bar{z})q(\theta)\}b_{n_0}+\tilde{F},\quad
&\theta=0,
\end{cases}\\
&\doteq  A_0 W + H(z,\bar{z},\theta),
\end{split}
\end{equation}
where
\begin{equation*}
H(z,\bar{z},\theta)=H_{20}(\theta)\frac{z^2}{2}+H_{11}(\theta)z\bar{z}+H_{02}(\theta)\frac{\bar{z}^2}{2}+\cdots.
\end{equation*}
Obviously,
\begin{equation*}
\begin{split}
H_{20}(\theta)&=\begin{cases}
-g_{20}q(\theta)b_{n_0}-\bar{g}_{02}\bar{q}(\theta)b_{n_0}, &\theta \in [-1,0),\\
-g_{20}q(0)b_{n_0}-\bar{g}_{02}\bar{q}(0)b_{n_0}+\tilde{F}_{20}, \quad & \theta=0,\\
\end{cases}\\
H_{11}(\theta)&=\begin{cases}
-g_{11}q(\theta)b_{n_0}-\bar{g}_{11}\bar{q}(\theta)b_{n_0}, &\theta \in [-1,0),\\
-g_{11}q(0)b_{n_0}-\bar{g}_{11}\bar{q}(0)b_{n_0}+\tilde{F}_{11}, \quad & \theta=0.\\
\end{cases}
\end{split}
\end{equation*}
Comparing the coefficients of \eqref{eq512} with the derived function
of \eqref{eq511}, we obtain
\begin{equation}\label{w_coefficients}
(A_0 -2i\omega_0\tau^*I)W_{20}(\theta)=-H_{20}(\theta),
\quad A_0 W_{11}(\theta)=-H_{11}(\theta).
\end{equation}
From \eqref{A} and \eqref{w_coefficients}, for $\theta\in[-1,0)$, we have
\begin{equation}\label{w solution}
\begin{split}
W_{20}(\theta)&=\frac{-g_{20}}{i\omega_{n_0}\tau^*}\begin{pmatrix} 1 \\
q_1 \end{pmatrix}e^{i\omega_{n_0}\tau^*\theta}b_{n_0}-\frac{\bar{g}_{02}}{3i\omega_{n_0}\tau^*}\begin{pmatrix} 1 \\
\bar{q}_1 \end{pmatrix}e^{-i\omega_{n_0}\tau^*\theta}b_{n_0}+E_1e^{2i\omega_{n_0}\tau^*\theta},\\
W_{11}(\theta)&=\frac{g_{11}}{i\omega_{n_0}\tau^*}\begin{pmatrix} 1 \\ q_1
\end{pmatrix}e^{i\omega_{n_0}\tau^*\theta}b_{n_0}-\frac{\bar{g}_{11}}{i\omega_{n_0}\tau^*}\begin{pmatrix} 1 \\
\bar{q}_1 \end{pmatrix}e^{-i\omega_{n_0}\tau^*\theta}b_{n_0}+E_2,
\end{split}
\end{equation}
where $E_1$ and $E_2$ can be obtained by setting $\theta=0$ in $H$, that is
\begin{equation}\label{E1,E2}
(A_0 -2i\omega_{n_0}^+\tau^*I)E_1e^{2i\omega_{n_0}^+\tau^*\theta}\mid_{\theta=0}+\tilde{F}_{20}=0,
\quad A_0 E_2\mid_{\theta=0}+\tilde{F}_{11}=0.
\end{equation}
The terms $\tilde{F}_{20}$ and $\tilde{F}_{11}$ are elements in the space $\mathscr{C}$
with
\begin{equation*}
\tilde{F}_{20}=\sum_{n=1}^{\infty}\langle \tilde{F}_{20}, \beta_n \rangle b_n, \quad
\tilde{F}_{11}=\sum_{n=1}^{\infty}\langle \tilde{F}_{11}, \beta_n \rangle b_n.
\end{equation*}
Denote
$$
E_1=\sum_{n=0}^{\infty}E_1^n b_n,~~E_2=\sum_{n=0}^{\infty}E_2^n b_n.
$$
Then from \eqref{E1,E2}, we have
\begin{eqnarray*}
	\begin{array}{rr}
		(A_0-2i\omega_{n_0}\tau^*I)E_1^n b_n e^{2i\omega_{n_0}\tau^*\theta}\mid_{\theta=0}&=-\langle \tilde{F}_{20}, \beta_n \rangle b_n,\\
		A_0 E_2^n b_n\mid_{\theta=0}&=-\langle \tilde{F}_{11}, \beta_n \rangle b_n,
	\end{array}
\end{eqnarray*}
where $n=0,1,\cdots$. Thus, $E_1^n$ and $E_2^n$ could be calculated by
\begin{eqnarray*}
	\begin{array}{l}
		E_1^n=\left(2i\omega_{n_0}\tau^*I-\displaystyle{\int}_{-1}^0e^{2i\omega_{n_0}\tau^*\theta}\text{d}\eta_n(0,\theta)\right)^{-1}\langle \tilde{F}_{20}, \beta_n \rangle,\\
		E_2^n=-\left(\displaystyle{\int}^0_{-1}\text{d}\eta_n(0,\theta)\right)^{-1}\langle \tilde{F}_{11}, \beta_n \rangle, \\
	\end{array}
\end{eqnarray*}
where
$$\langle \tilde{F}_{20}, \beta_n \rangle=\begin{cases}\frac{1}{\sqrt{l\pi}}\hat{F}_{20}, &n_0\neq0,~n=0,\\
\frac{1}{\sqrt{2l\pi}}\hat{F}_{20}, &n_0\neq0,~n=2n_0,\\
\frac{1}{\sqrt{l\pi}}\hat{F}_{20}, &n_0=0,~n=0,\\
0, &other,\end{cases},~
\langle \tilde{F}_{11}, \beta_n \rangle=\begin{cases}\frac{1}{\sqrt{l\pi}}\hat{F}_{11}, &n_0\neq0,~n=0,\\
\frac{1}{\sqrt{2l\pi}}\hat{F}_{11}, &n_0\neq0,~n=2n_0,\\
\frac{1}{\sqrt{l\pi}}\hat{F}_{11}, &n_0=0,~n=0,\\
0, &other,\end{cases}$$
with
\begin{eqnarray*}
	\begin{array}{l}
		\hat{F}_{20}=\left(\begin{array}{c} 2 r q_1 e^{-i\omega_{n_0}\tau^*}+\cfrac{2}{(1+m^*)^2}e^{-i\omega_{n_0}\tau^*}-\cfrac{2 m^*}{(1+m^*)^3}e^{-i2\omega_{n_0}\tau^*}\\
			-2 \gamma^{-1}q_1\end{array}\right),\\
		\hat{F}_{11}=\left(\begin{array}{c} r(q_1e^{-i\omega_{n_0}\tau^*}+\bar{q}_1e^{i\omega_{n_0}\tau^*})+\cfrac{2}{(1+m^*)^3}\\
			-\gamma^{-1}( q_1+\bar{q}_1)\end{array}~\right).
	\end{array}
\end{eqnarray*}
Hence, $g_{21}$ could be represented explicitly.

Denote
\begin{equation}\label{c1}
\begin{split}
&c_1(0)=\frac{i}{2\omega_{n_0}\tau^*}(g_{20}g_{11}-2|g_{11}|^2-\frac{1}{3}|g_{02}|^2)+\frac{1}{2}g_{21},\\
&\mu_2=-\frac{\text{Re}(c_1(0))}{\tau^*\text{Re}(\lambda'(\tau^*))},~~~
\beta_2=2\text{Re}(c_1(0)),\\
&T_2=-\frac{1}{\omega_{n_0}\tau^*}(\text{Im}(c_1(0))+\mu_2(\omega_{n_0}+\tau^*\text{Im}(\lambda'(\tau^*))).
\end{split}
\end{equation}

By the general results of Hopf bifurcation theory \cite{Hassard-1981}, the properties of Hopf bifurcation can be determined by the parameters in \eqref{c1}.
$\beta_2$ determines the stability of the bifurcating
periodic solutions: the bifurcating periodic solutions are orbitally asymptotically
stable(unstable) if $\beta_2<0(>0)$;
$\mu_2$ determines the direction of the Hopf
bifurcation: if $\mu_2>0(<0)$, the direction of the Hopf bifurcation is forward (backward), that is the bifurcating
periodic solutions exist when $\tau>\tau^*(<\tau^*)$;
$T_2$ determines the period of the bifurcating periodic solutions: when $T_2>0(<0)$, the
period increases(decreases) as the $\tau$ varies away from $\tau^*$.

From Lemma \ref{trans_tau} in Section \ref{stability}, we know that $\text{Re}(\lambda'(\tau^*))>0$. Combining with above discussion, we obtain the following theorem.
\begin{theorem}\label{Hopf_direction}
	If $\text{Re}(c_1(0))<0(>0)$, then the bifurcating periodic solutions exists for $\tau>\tau^*(<\tau^*)$ and are orbitally
	asymptotically stable(unstable).
\end{theorem}

\section{Simulations}\label{simulations}

In this section, we shall show some simulations to illustrate our theoretical results. Let $l=1$, and choose
$$
\gamma=0.5,~~~d=1.0,~~~\alpha=0.10,~~~r=2.
$$
Since $0<\alpha<1<r<\alpha^{-1}$, then $E_*(m^*, a^*)$ is the only positive equilibrium with $(m^*, a^*)=(0.1250,0.4444)$. One can easily verify
that \textsc{(H1)} $\sim$ \textsc{(H3)} are satisfied. By a simple calculation, we also obtain that Eq.\eqref{omega_2} has a positive root only for $n=0$, and
$$
\omega_0\approx 0.3253,~~~~   \tau^*\approx 2.3545.
$$

Furthermore, we have $c_1(0)\approx-2.28261-23.9865i$, which means $\beta_2<0$, $\mu_2>0$.
From Theorem \ref{Hopf_stability} and \ref{Hopf_direction}, the
positive equilibrium $E_*(0.1250,0.4444)$ is locally asymptotically stable when $\tau\in[0, \tau^*)$ (see Fig.\ref{fig1}), moreover, system \eqref{eq_ma_tau} undergoes a Hopf bifurcation at $\tau=\tau^*$, the direction of the Hopf bifurcation is forward and bifurcating periodic solutions are orbitally asymptotically stable (see Fig.\ref{fig2}).

If we choose
$$
\gamma=0.5,~~~d=1.0,~~~\alpha=0.10,~~~r=0.5, ~~~\tau=2.
$$
Here $r=0.5\in(0,1)$, from Remark \ref{E_0}, we know that the boundary equilibrium
$E_0(0,1)$ is locally asymptotically stable (see Fig.\ref{fig3}).

Fig.\ref{fig1} $\sim$ Fig.\ref{fig3} show the dynamics of system \eqref{eq_ma_tau} near the positive constant steady state.
Fig.\ref{fig1} shows a stable positive constant steady state when $\tau<\tau^*$ and $r>1$, which represents coexistence of both species (mussel and algae) biologically.
This stability will be broken when $\tau$ increases and passes through the critical value $\tau^*$, which is accompanied by a spatially homogeneous periodic solution corresponds to a periodic oscillation in populations of mussel and algae, see Fig.\ref{fig2}. This periodicity is common in predator-prey systems \cite{CSW-2013, Shen-2018}. Fig.\ref{fig3} shows
the prey-only homogeneous steady state under the condition $0<r<1$, which corresponds to bare sediment with no mussel biomass.The initial conditions are given by $m_0(x,t)=m^*+0.1\cos2x$, $a_0(x,t)=a^*-0.1\cos2x$, $(x,t)\in[0,\pi]\times[-\tau,0]$.

Our results suggest that the positive constant steady state will lose its stability when $\tau$ passes through some critical values, and there will be periodic oscillations in populations of species. We have tried a large number of sets of parameters, but we did not find any set that would allow a nonhomogeneous periodic solution bifurcating from the steady state under the assumption \textsc{(H1)}$\sim$\textsc{(H3)}.
Biologically, for mussels and algae species living at the same depth, if the digestion period $\tau$ is greater than the critical value $\tau^*$, the population will have a periodic oscillation over time with their spatial distribution is uniform.

\begin{figure}[htp]
\begin{multicols}{2}
\begin{center}
\subfigure{\includegraphics[width=3in]{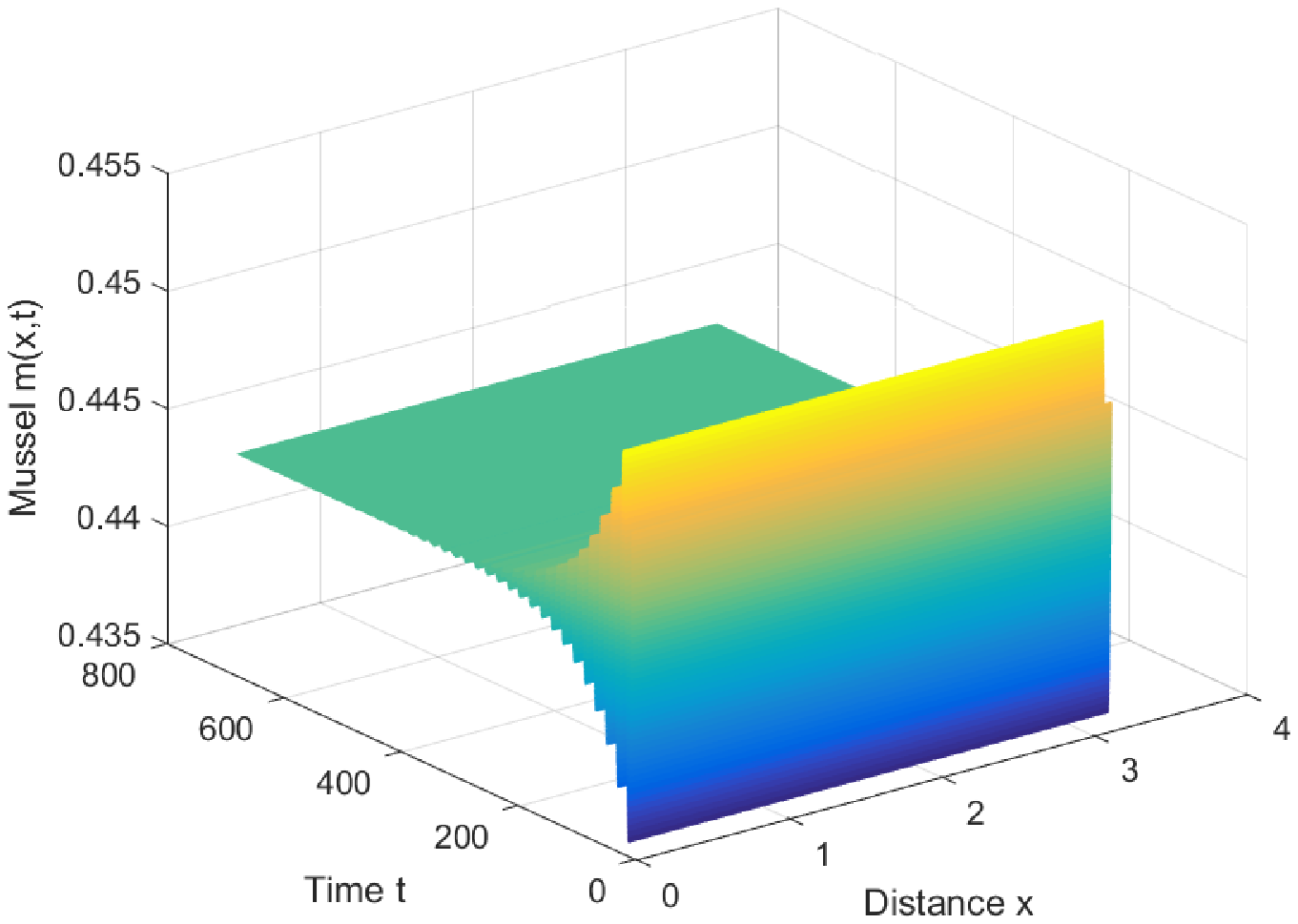}}\vspace{-0.5cm}

\subfigure{\includegraphics[width=3in]{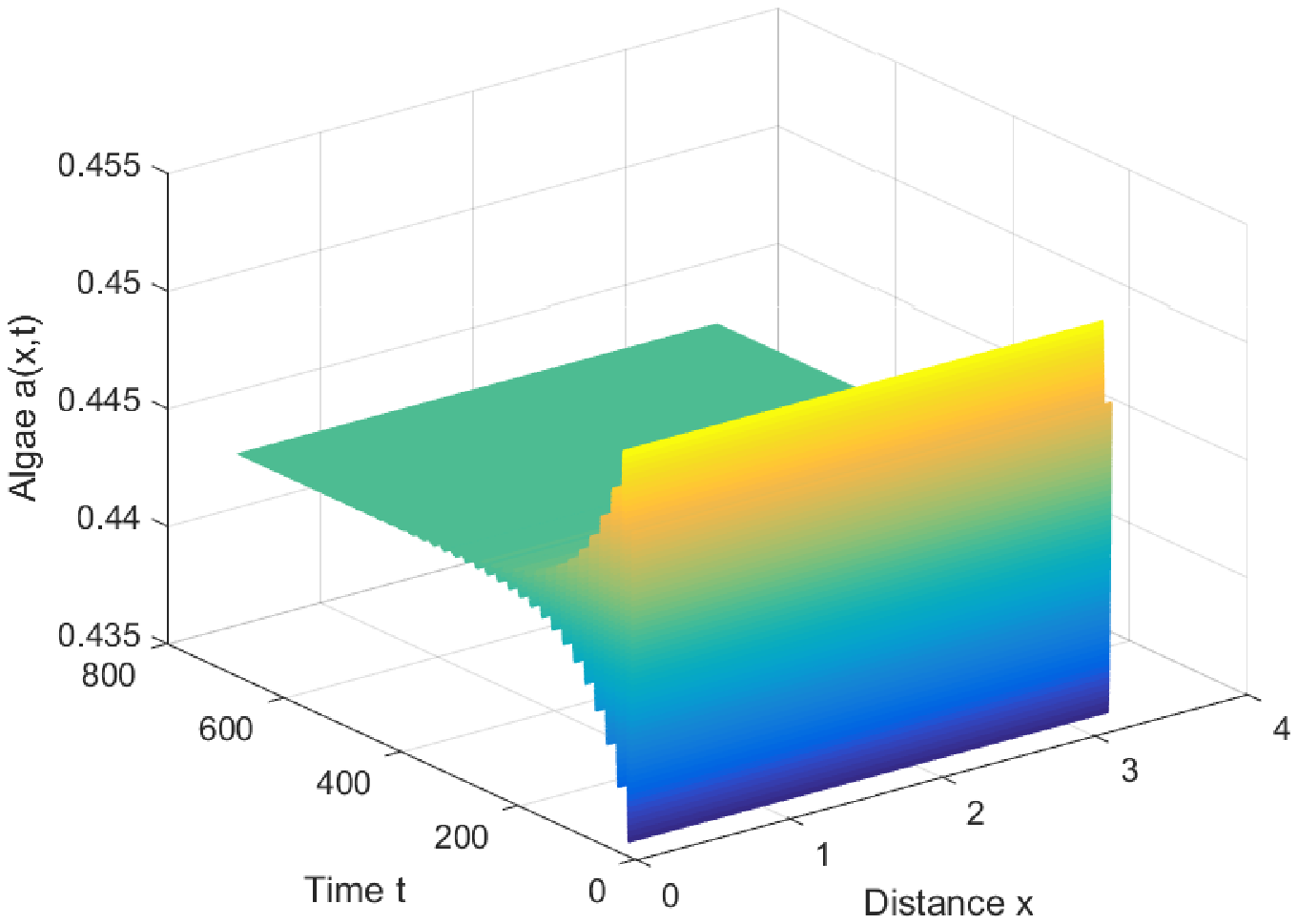}}
\end{center}
\begin{center}
\subfigure{\includegraphics[width=3.5in,height=2.5in]{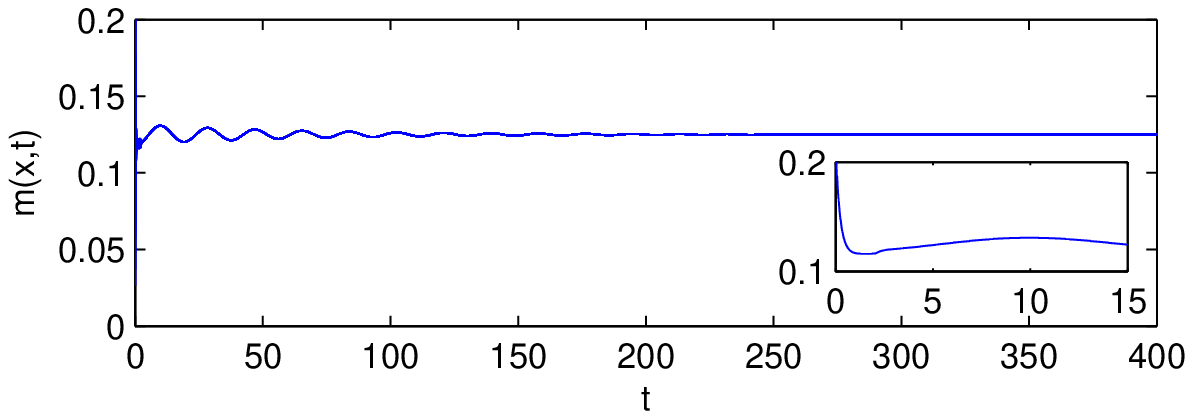}}\vspace{-3cm}
\subfigure{\includegraphics[width=3.5in,height=2.5in]{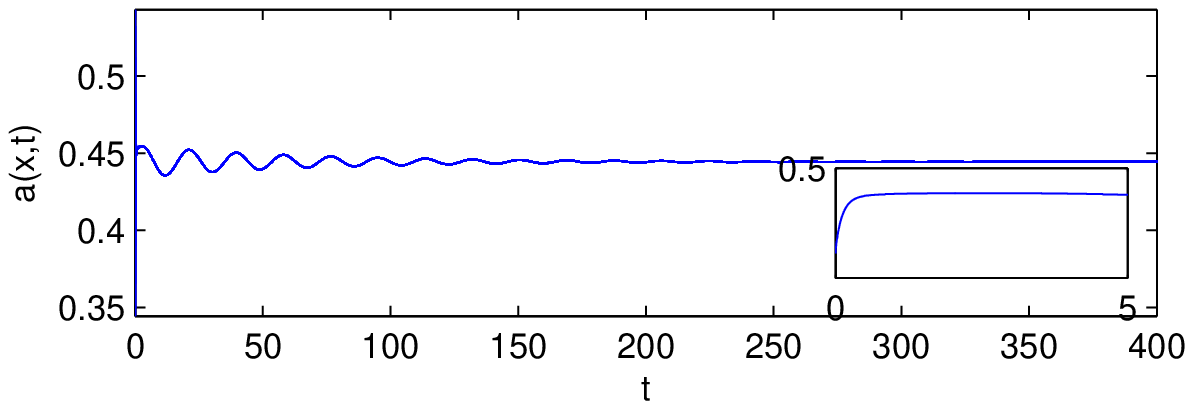}}\vspace{-3cm}
\subfigure{\includegraphics[width=3.5in,height=2.5in]{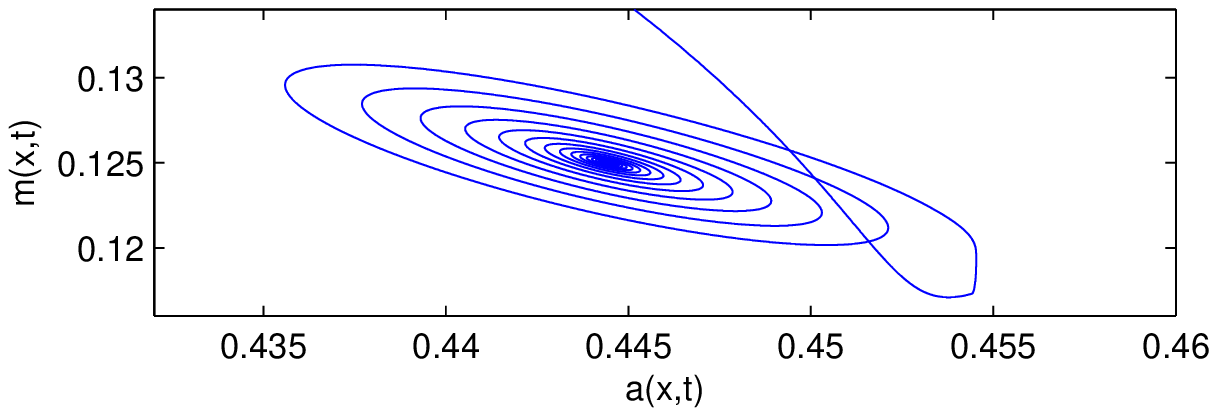}}
\end{center}
\end{multicols}
\vspace{-2cm}
 \caption{The positive equilibrium is locally asymptotically stable when $\tau\in[0, \tau^*)$, where $\tau=2<\tau^*\approx2.3545$.}\label{fig1}
\end{figure}

\begin{figure}[htp!]
\begin{multicols}{2}
\begin{center}
\subfigure{\includegraphics[width=3in]{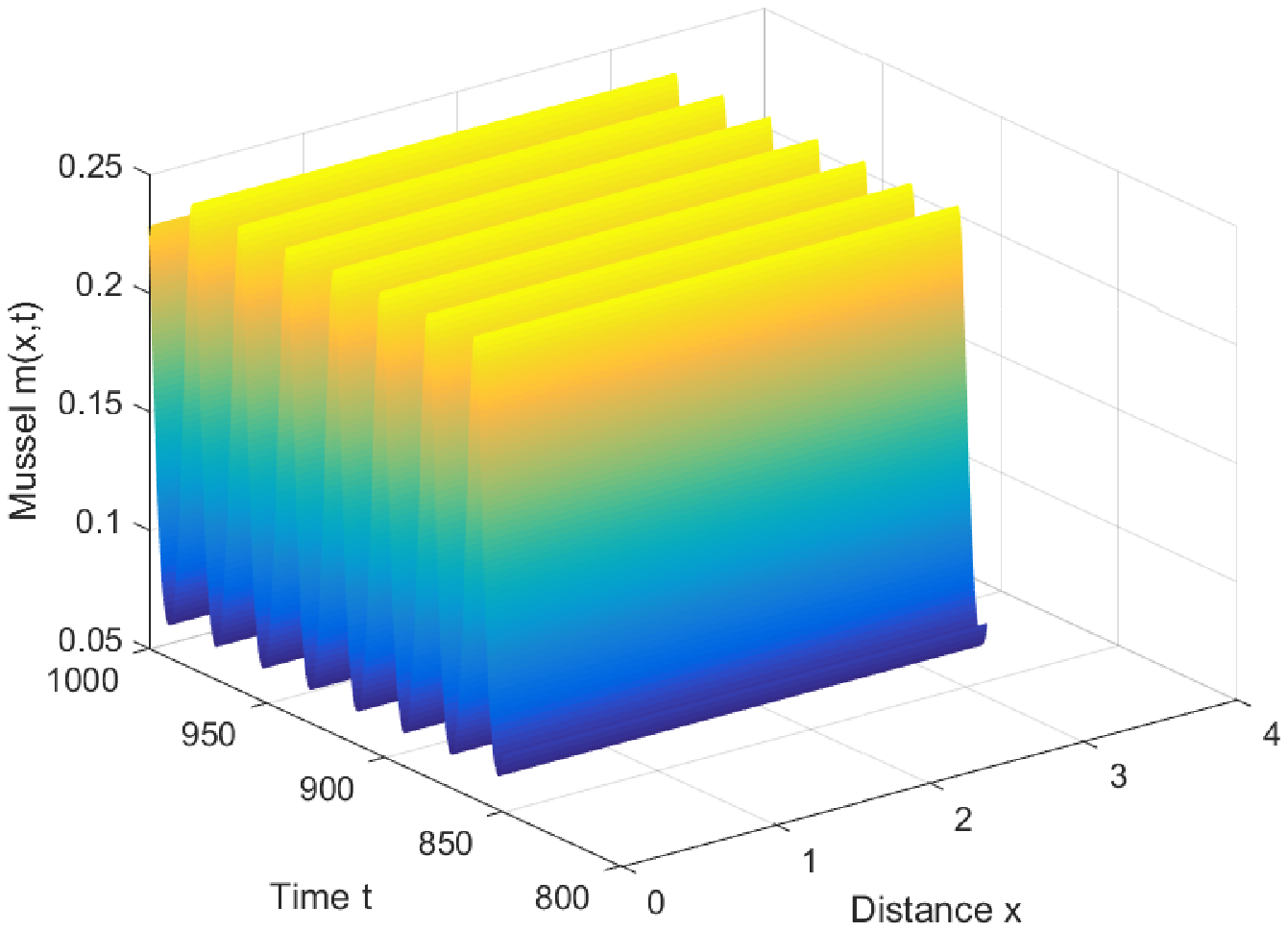}}\vspace{-0.5cm}
\subfigure{\includegraphics[width=3in]{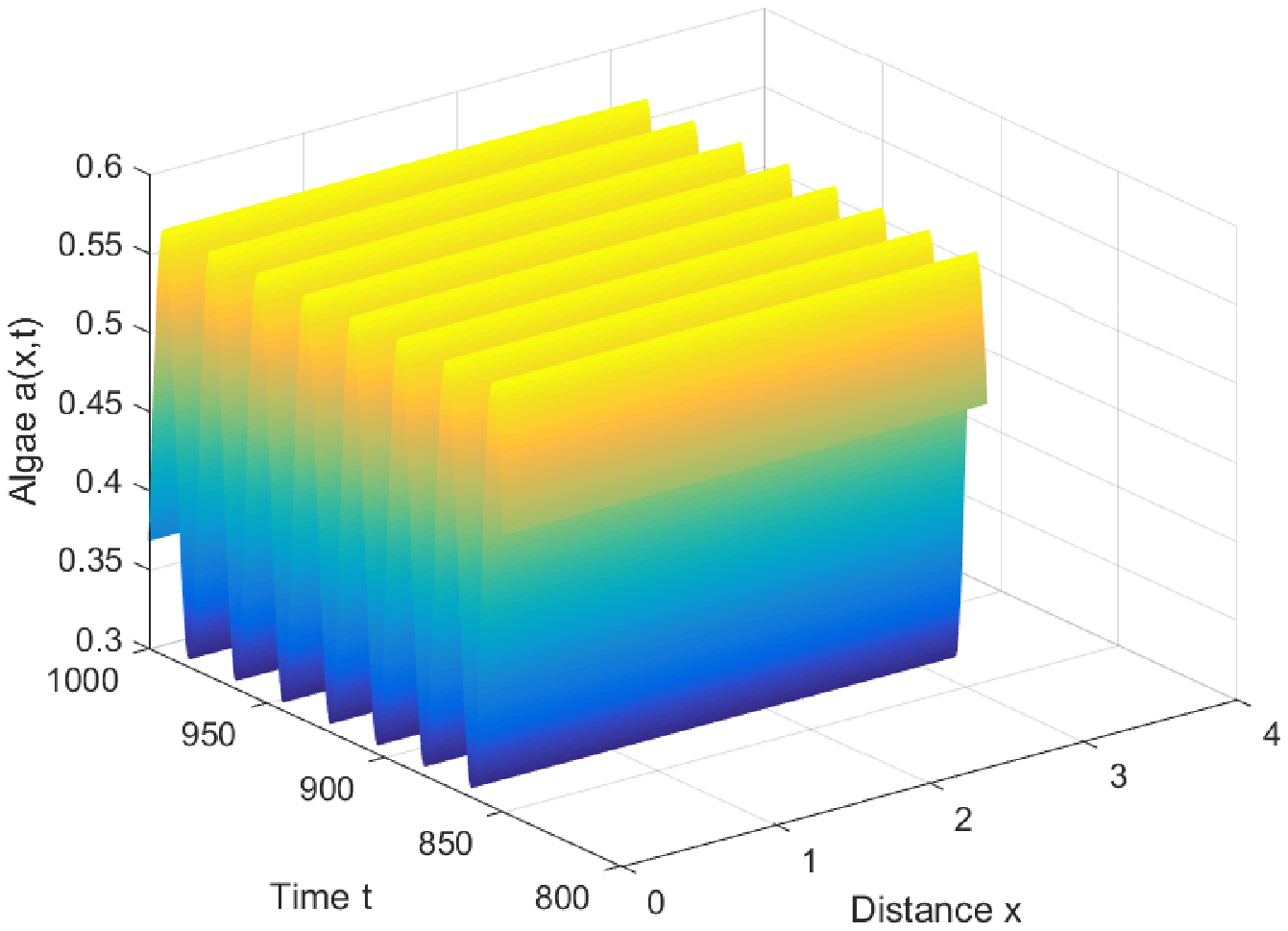}}
\end{center}
\begin{center}
\subfigure{\includegraphics[width=3.4in,height=2.5in]{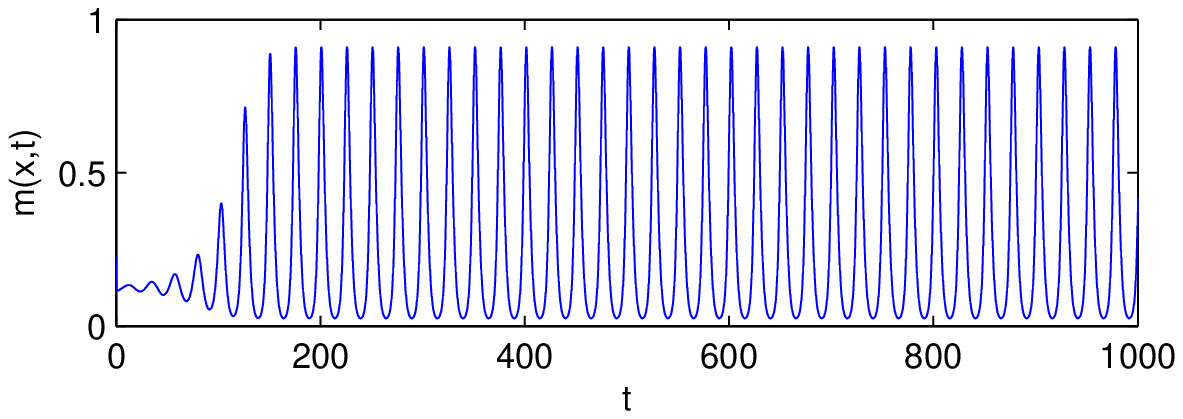}}\vspace{-3cm}
\subfigure{\includegraphics[width=3.4in,height=2.5in]{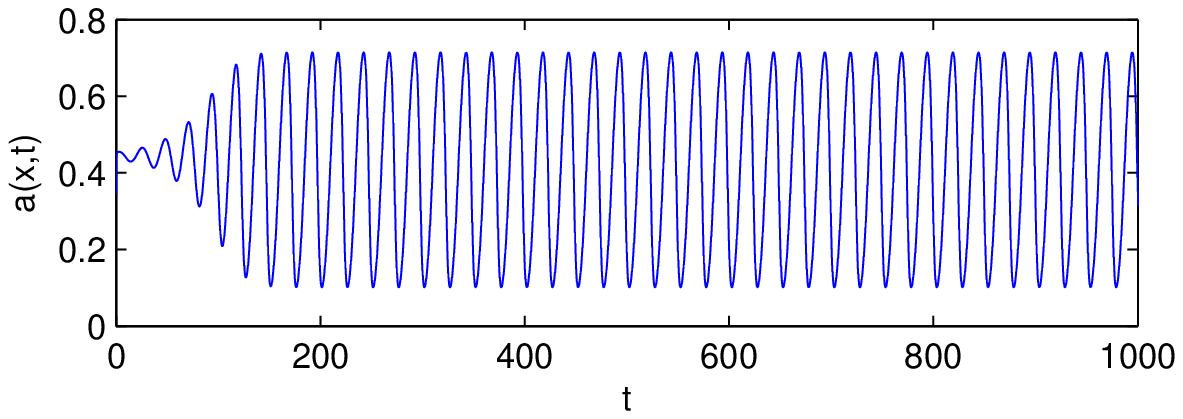}}\vspace{-3cm}
\subfigure{\includegraphics[width=3.4in,height=2.5in]{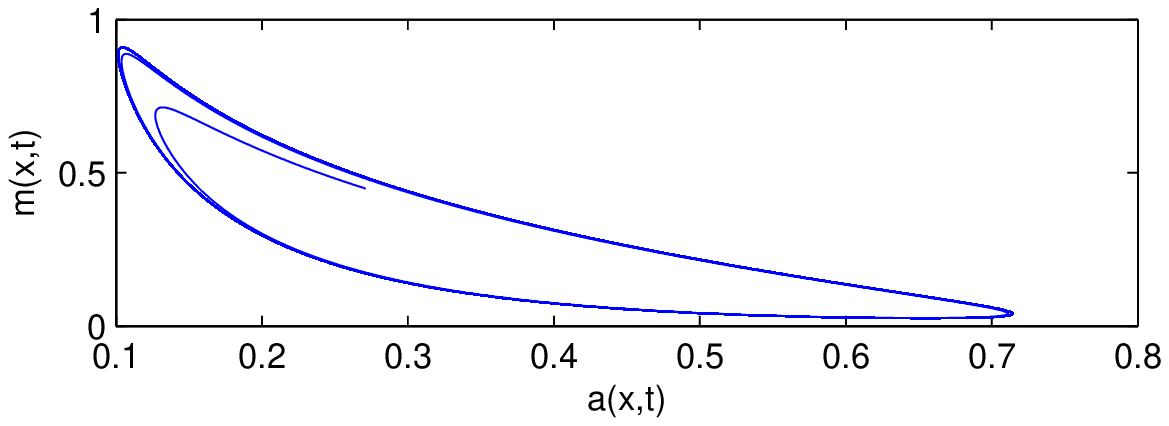}}
\end{center}
\end{multicols}
\vspace{-2cm}
 \caption{The bifurcating periodic solution is orbitally asymptotically stable, where $\tau=3.6>\tau^*\approx2.3545$.}\label{fig2}
\end{figure}

\begin{figure}[htp]
\centering
\begin{multicols}{2}
\begin{center}
\subfigure{\includegraphics[width=3in]{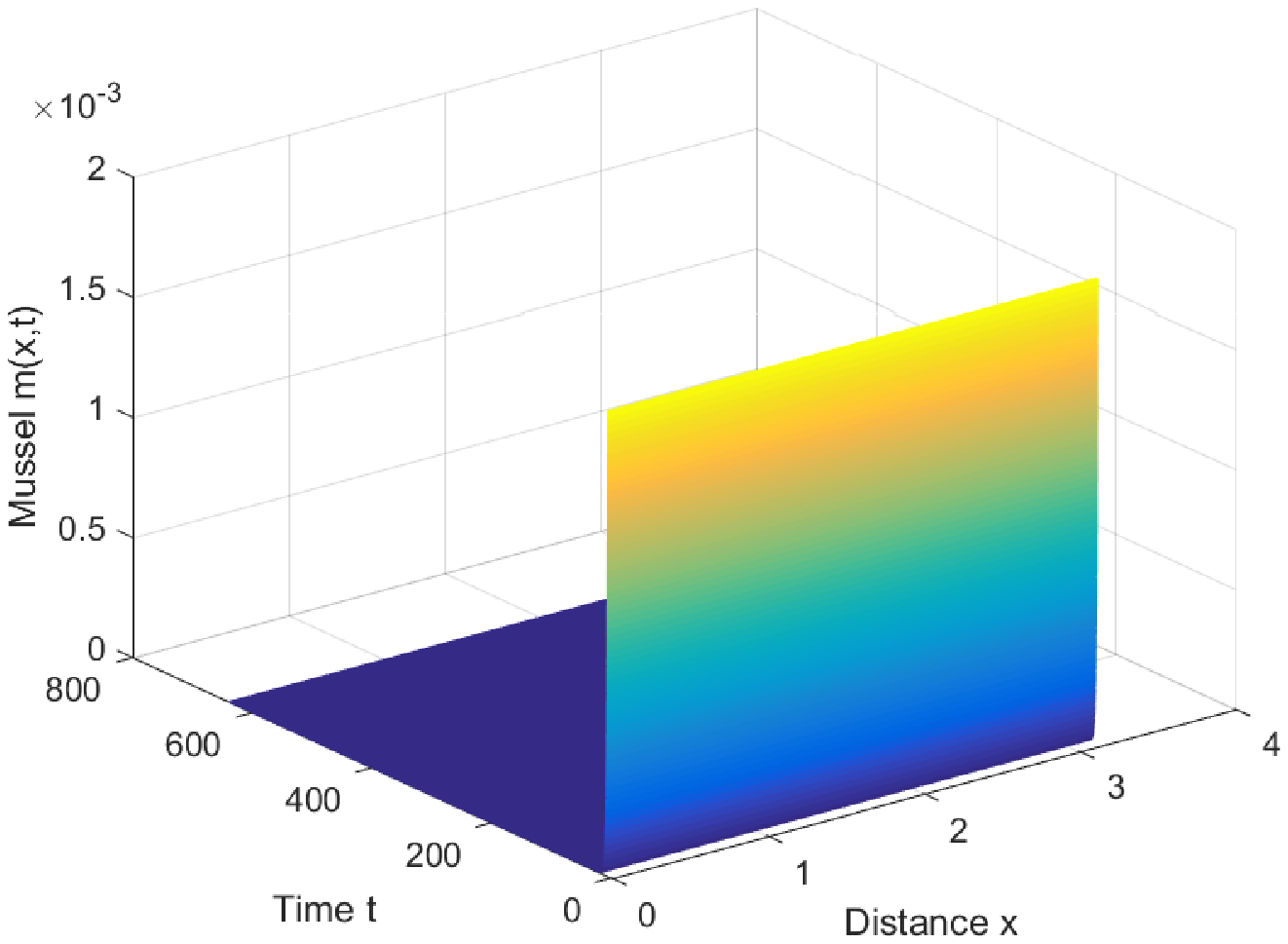}}\vspace{-0.5cm}
\subfigure{\includegraphics[width=3in]{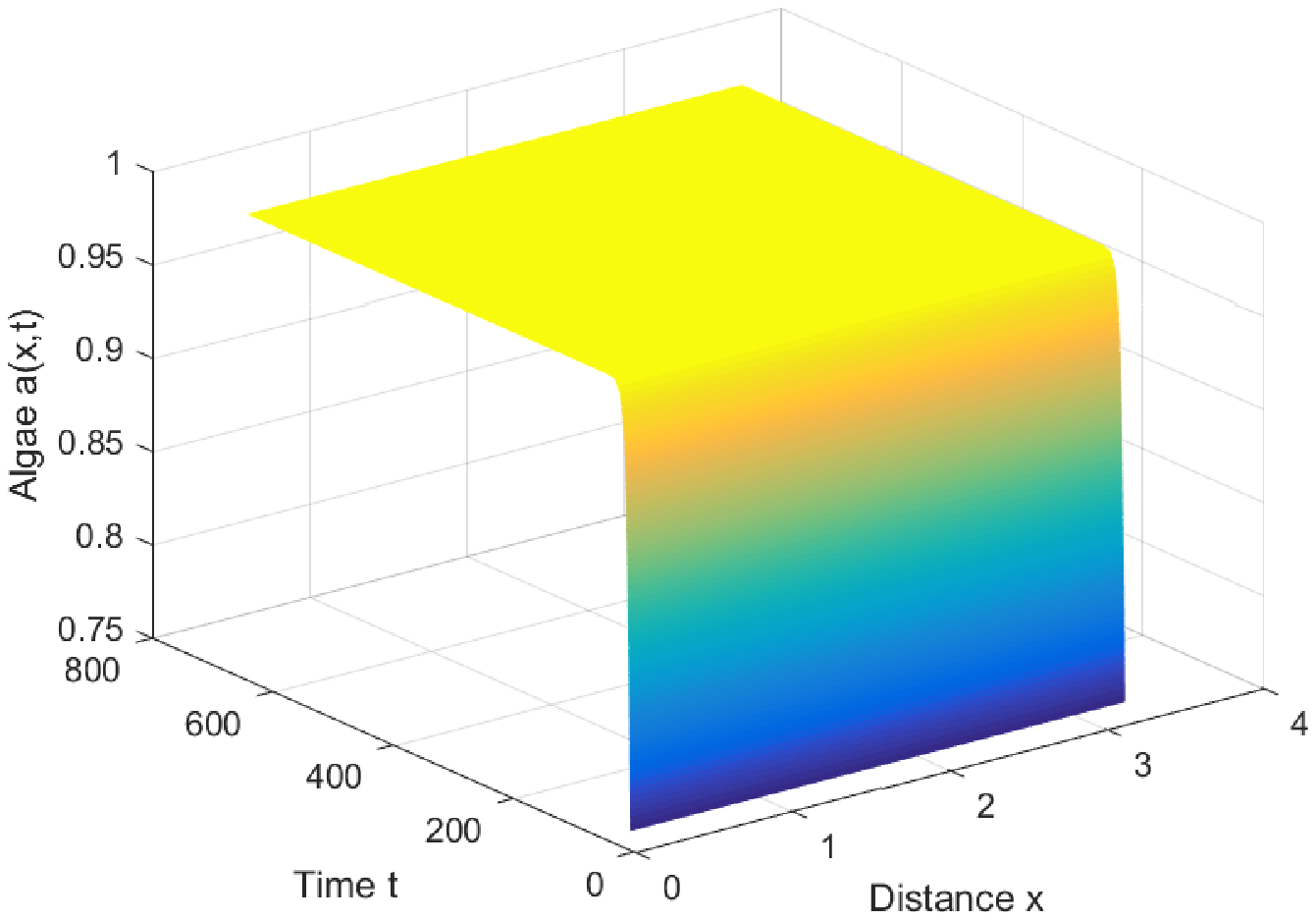}}
\end{center}

\begin{center}
\subfigure{\includegraphics[width=3.4in,height=2.5in]{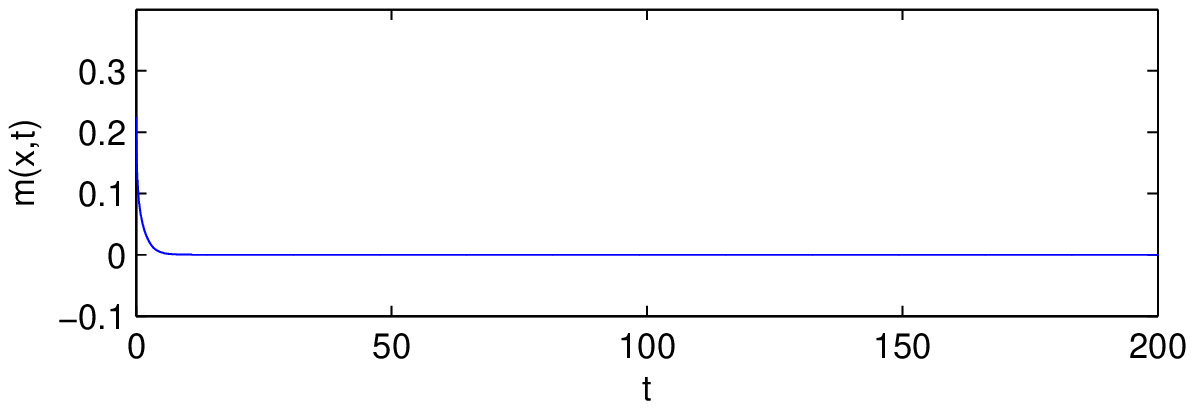}}\vspace{-3cm}
\subfigure{\includegraphics[width=3.4in,height=2.5in]{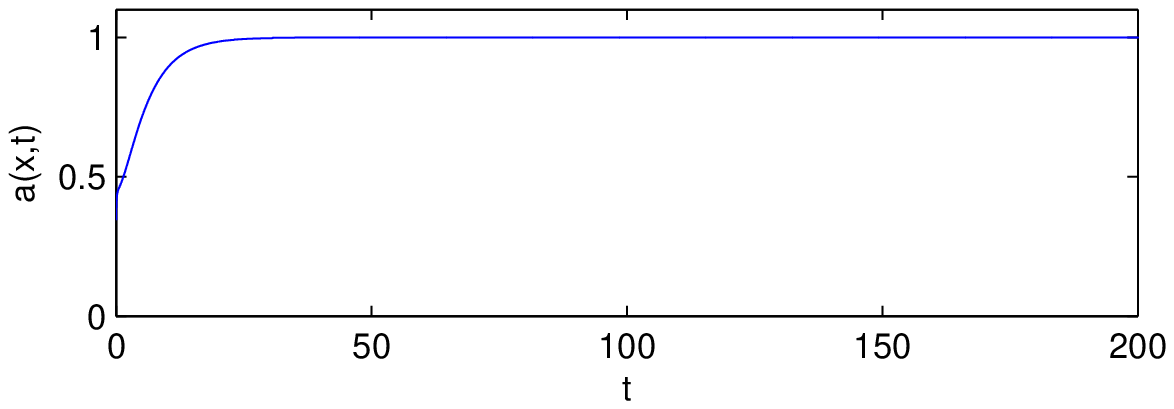}}\vspace{-3cm}
\subfigure{\includegraphics[width=3.4in,height=2.5in]{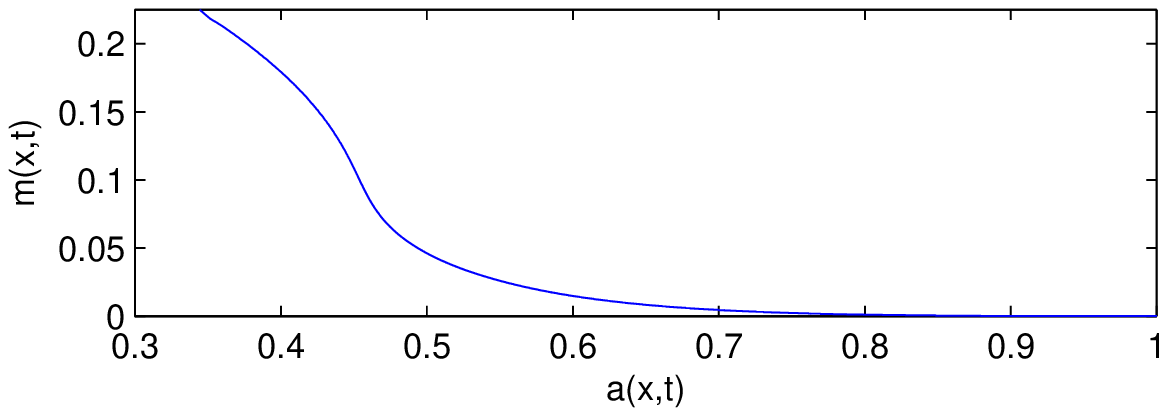}}
\end{center}
 \end{multicols}
 \vspace{-2cm}
  \caption{The axial equilibrium $E_1(0,1)$ is locally asymptotically stable.}\label{fig3}
\end{figure}

{\bf Acknowledgements} The authors are grateful to the anonymous referees for their
helpful comments and valuable suggestions which have improved the presentation of the paper.

\end{document}